\documentclass[12pt]{amsart}
\usepackage{epsf,amscd,graphicx,amssymb,mathrsfs}

\pdfoutput=1
\usepackage{amscd,amssymb,mathrsfs}

\usepackage{lpic}
\usepackage{amsmath} 
\usepackage{graphics}
\usepackage{rotating}
\usepackage{url}
\usepackage{pinlabel}  

\usepackage{floatrow}

\usepackage{marginnote}

\theoremstyle{plain}

\newtheorem{thm}{Theorem}[section]
\newtheorem{lem}[thm]{Lemma}

\newtheorem{prop}[thm]{Proposition}
\newtheorem{cor}[thm]{Corollary}
\theoremstyle{remark}

\theoremstyle{definition}
\newtheorem{defn}[thm]{Definition}

\renewcommand{\o}{\mathsf}   

\newcommand{\R}{\mathbb R}
\newcommand{\Ss}{\mathbb{S}}   

\newcommand{\V}{\mathsf{V}}
\newcommand{\Lspace}{\mathsf{E}}
\newcommand{\A}{\mathscr{A}}   
\newcommand{\W}{\mathscr{W}}

\newcommand{\vu}{\mathsf{u}}
\newcommand{\vv}{\mathsf{v}}
\newcommand{\vw}{\mathsf{w}}

\newcommand{\vn}{\mathsf{n}}
\newcommand{\vp}{\mathsf{p}}
\newcommand{\be}{\mathsf{e}}

\newcommand{\bbb}{\o{b}} 
\newcommand{\SOto}{\o{SO}(2,1)}
\newcommand{\cD}{{\mathcal{D}}}
\newcommand{\bg}{\mathbf{g}}
\newcommand{\vve}{\vec{e}}
\newcommand{\Mat}{\mathsf{Mat}}

\newcommand{\Axis}{\mathsf{Axis}}
\newcommand{\hAxis}{\widehat{\mathsf{Axis}}}

\newcommand{\cpo}{\bC\P^1}
\newcommand{\rpo}{\R\P^1}

\renewcommand{\P}{\o{P}}   

\newcommand{\Det}{\mathsf{Det}}

\newcommand{\tSigma}{\tilde{\Sigma}}
\newcommand{\dev}{\mathsf{dev}}
\newcommand{\He}{\mathscr{H}}
\newcommand{\Isom}{\mathsf{Isom}}
\newcommand{\Uu}{\mathbb{U}}
\newcommand{\bnu}{\boldsymbol \nu}
\newcommand\bC{{\mathbb C}}
\newcommand\bZ{{\mathbb Z}}

\newcommand\bH{{\mathbb{H}}}

\renewcommand{\L}{\mathbb{L}} 

\newcommand\GL{{\rm GL}}
\newcommand\SL{{\rm SL}}
\newcommand\PGL{{\rm PGL}}

\newcommand\SO{{\rm SO}}
\newcommand\SI{{\bf S}}

\newcommand\clo{{\rm Cl}} 
\newcommand\bdd{{\mathbf d}}  

\newcommand\ra{\rightarrow}

\newcommand\emp{\emptyset}
\newcommand\eps{\epsilon}

\newcommand\vth{\vartheta}

\newcommand\Aff{{\mathsf{Aff}}}
\newcommand\ovl{\overline}
\newcommand\Aut{{\mathsf{Aut}}}
\newcommand\Idd{{\rm I}}

\newcommand\tri{\triangle}

\newcommand{\rp}[1]{\R\P^#1}
\newcommand\rpn{\rp{n}}
\newcommand\rptwo{\rp{2}}
\newcommand\rpthree{\rp{3}}
\newcommand\rpone{\rp{1}}

\renewcommand{\Join}{\mathsf{Join}}

\newcommand{\llrrparen}[1]{
  \left(\mkern-4mu\left(#1\right)\mkern-4mu\right)}

\newcommand{\bx}{\mathbf{x}}
\newcommand{\bb}{\mathbf{b}}

\begin{document}

\title[Margulis spacetimes]{Topological tameness of Margulis spacetimes}

\author[Choi]{Suhyoung Choi}
\address{Department of Mathematical Sciences \\ KAIST \\
Daejeon 34141, Republic of Korea \\
}
\email{schoi@math.kaist.ac.kr}
\author[Goldman]{William Goldman}
\address{Department of Mathematics \\ University of Maryland 
\\College Park, MD 20742-4015, USA 
}
\email{wmg@math.umd.edu}
\date{\today}
\thanks{The first author was supported by the National Research Foundation of Korea(NRF) grant 
funded by the Korea government(MEST) (No.2010-0027001).
The second author gratefully acknowledges partial support from 
National Science Foundation grants  DMS-070781 and  DMS-1065965.}

\dedicatory{\smallskip
\begin{center}
{Dedicated to the memory of Bill Thurston, our mentor}
\end{center}
}

\begin{abstract}
We show that  Margulis spacetimes without parabolic holonomy elements are topologically tame. 
A Margulis spacetime is the quotient of the $3$-dimensional Minkowski space by  
a free proper isometric action of the free group of rank $\geq 2$.
We will use our particular point of view that the Margulis spacetime 
is a manifold-with-boundary with an $\rpthree$-structure
in an essential way.
The basic tools are a bordification by a closed $\rptwo$-surface with  
free holonomy group, and the work of Goldman, Labourie, and Margulis on 
geodesics in the Margulis spacetimes and $3$-manifold topology.

\end{abstract}




\nopagebreak

\maketitle
\tableofcontents





\section{Introduction}

A {\em complete affine manifold\/} is a quotient $M$ of an $n$-dimensional real affine 
space $\Lspace$ 
by a discrete group $\Gamma$ of affine transformations acting properly.
Equivalently, $M$ is a manifold with a geodesically flat torsion-free affine  connection.  

Complete affine $3$-manifolds with solvable fundamental groups were classified by 
Fried-Goldman~\cite{FG}, and these include all cases where $M$ is compact.
In all other cases, $M$ admits a parallel Lorentzian structure, 
and $\pi_1(M) \cong \Gamma$ is a free group. 
The existence of these manifolds was first demonstrated by Margulis \cite{Mar}, 
and later clarified by Drumm \cite{Drumm_jdg} using {\em crooked planes.\/}
In particular Drumm's examples are {\em tame,\/} 
that is, homeomorphic to the interiors of compact manifolds-with-boundary, 
which in this case are {\em closed solid handlebodies.\/} 

Charette, Drumm, Goldman, and Labourie \cite{Drumm_jdg}, \cite{DG}, \cite{DG2}
made contributions to understanding such spacetimes with parallel Lorentzian structures
culminating in recent work \cite{GLM} and \cite{GL}.  The complete classification of such spacetimes is currently going on with the work of
Charette, Drumm, Goldman, Labourie, Margulis and Minsky \cite{CDG}, \cite{CDG2}, \cite{GLM}, and \cite{GLMM}.

The theory of compactifying open manifolds goes back to 
Browder, Levine and Livesay \cite{BLL} and Siebenmann \cite{Sieb}. 
For $3$-manifolds, Tucker \cite{Tu} and Scott and Tucker \cite{ST} 
made an initial clarification. (See also Meyers \cite{Me}. )
As a note, we state that complete 
hyperbolic $3$-manifolds with finitely generated fundamental groups were shown to be tame
by Agol \cite{Agol}, Calegari, and Gabai \cite{CalGa}. 
For examples of nontame finitely generated $3$-manifolds, see the examples of Scott-Tucker \cite{ST} and 
Canary-Minsky \cite{CanMin}, and Example 4 of Souto \cite{Souto}. 
Also, see Bowditch \cite{Bow} for details on this very large topic. 
We won't elaborate on this as the methods are completely  
different. Our proof is closer to the proof of
the tameness of geometrically finite hyperbolic $3$-manifolds due 
to  Marden \cite{Marden} (see  Thurston~\cite{Thnote} and  Epstein and Marden \cite{Epstein}. 
See also Bonahon \cite{Bonahon}).


Let $V$ be a vector space over $\R$. 
The {\em projectivization}  $\P(V)$ is defined as the quotient space 
\[V -\{O\}/\sim \hbox{ where } 
\vv \sim \vw \hbox{ if and only if } \vv= s \vw \hbox{ for } s \in \R -\{0\}.\]
Define $\Ss(V)$ as 
\[V -\{O\}/\sim \hbox{ where } 
\vv \sim \vw \hbox{ if and only if } \vv= s \vw \hbox{ for } s > 0.\]
A {\em projective subspace} of $\SI^{n}$ is the image of $\Ss(W)$ of the vector subspace $W$ of $\R^{n+1}$
under the projection $\R^{n+1} -\{O\} \ra \Ss(\R^{n+1}) = \SI^{n}$. 

Let $\V$ denote a $3$-dimensional real vector space 
with a Lorentzian inner product of signature $2,1$ and a fixed orientation.
The spacetime 
$\Lspace$ is an affine space with the underlying vector space $\V$.
A {\em Lorentzian isometry} is an automorphism of $\Lspace$ 
preserving the Lorentzian inner product. 
Denote the group of orientation-preserving Lorentzian isometries
by $\Isom^+(\Lspace)$. 
$\P(\V)$ is identical with the real projective plane $\rptwo$. 
Here $\PGL(3, \R)$ acts faithfully on $\rptwo$ as the group of 
collineations. 
Recall that we embed the hyperboloid model $\bH^2$ of the hyperbolic plane 
to the disk that is the interior of the conic in $\rptwo$ determined by null vectors. 

Define  the {\em sphere of directions}  $\SI^2_{\infty}:=\Ss(\V)$ double-covering $\rptwo$. 
Then the image $\Ss_+$ of the space of future timelike vectors 
identifies with the hyperbolic $2$-plane $\bH^2$, 
which is basically the Beltrami-Klein model of the hyperbolic plane.
Let $\Ss_-$ denote the subspace of $\Ss$ corresponding to past timelike vectors.
The group of orientation preserving linear maps of $\V$ is denoted by $\SOto$, a Lie group
with two components one of which is the identity component $\SOto^o$.
The linear group $\SOto^o$ acts faithfully  on $\bH^2 = \Ss_+$ as the orientation-preserving isometry group 
and $\SOto$ acts so on $\Ss_+ \cup \Ss_-$ 
and acts on $\SI^2_{\infty}$ projectively where the action 
is induced from that on $\V$.


Since $\SOto$ injects to $\PGL(3, \R)$ under projection $\GL(3, \R) \ra \PGL(3, \R)$, 
we consider that $\SOto$ acts on $\Ss_+ = \bH^2$ a subset of $\rptwo$ 
as well in a {\em projective manner}. 
Additionally, $\SOto^o$ acts on $\Ss_+$ {\em honestly} as a subset of $\SI^2_{\infty}$. 

Denote by 
\[  \Isom^+(\Lspace) \xrightarrow{\L} \SOto\] 
obtained by taking the linear part of an affine isometry of $\Lspace$.
We also denote by \[{\L}': \Isom^+(\Lspace) \stackrel{\L}{\ra} \SOto \hookrightarrow \PGL(3, \R)\] 
with the image acting on $\rptwo$. 
The identity component $\Isom^{+}_{0}(\Lspace)$ of $\Isom^{+}(\Lspace)$ acts on $\Ss_{+}$ and $\Ss_{-}$ respectively and
where $\L(\Isom^{+}_{0}(\Lspace))\subset \SOto^{o}$ holds. 
%

A linear Lorentzian isometry is {\em hyperbolic\/} if its eigenvalues are distinct and real.
Hyperbolic isometries comprise an open subset of $\SOto$, and in a suitable basis
represented by matrices
\[ 
\bmatrix \pm e^l & 0 & 0 \\  0 & 1 & 0 \\  0 & 0 & \pm e^{-l} \endbmatrix \]
(where both signs agree). If all the eigenvalues are positive, then 
such an isometry is said to be {\em positive hyperbolic.\/}

Suppose that $\Gamma$ is a finitely generated orientation-preserving Lorentzian isometry group acting 
freely and properly on $\Lspace$. We assume that $\Gamma$ is not amenable
(that is, not solvable). 
Since $\Lspace/\Gamma$ is aspherical, the fundamental group $\Gamma$ is torsion-free
by Lemma 9.4 of \cite{Hempel}.  
\begin{itemize} 
\item $\L|\Gamma$ is faithful and $\L(\Gamma)$ is discrete by \S 2.1.2 of Fried-Goldman \cite{FG}.
\item By Mess \cite{Mess}, $\Gamma$ is a free group of rank $\geq 2$. 
\end{itemize}
Thus, $\Gamma$ injects under $\L$ to $\L(\Gamma) \subset \SOto$ acting properly 
and freely on $\Ss_+ \cup \Ss_-$. 
We obtain an exact sequence 
\[ 1 \ra \Gamma' \ra \Gamma \ra J \ra 1\]
where $\Gamma'$ is the subgroup of $\Gamma$ of index $\leq 2$ so that 
${\L}(\Gamma') \subset \SO(2, 1)^o$ and $J$ is either $\bZ/2\bZ$ or is the trivial group. 
Since $\Gamma$ acts freely on $\Lspace$, we obtain that $\L(\Gamma)$ cannot contain $-\Idd$.
Hence, $\L'$ is faithful.
(See also \S 3 of Goldman-Labourie \cite{GL} for details.)
Thus, we restrict ourselves to the group actions of free groups of rank $\geq 2$ in this paper.
Then $\Lspace/\Gamma$ is said to be a {\em Margulis spacetime}, and an element of $\Gamma$ is 
said to be a {\em holonomy}. 

Note that $\bH^2/\L'(\Gamma)$ is not orientable if $\L(\Gamma)$ is not 
a subset of $\SOto^o$. 
However, $\Ss_+ \cup \Ss_-/{\L}(\Gamma)$ is an orientable surface 
double-covering $\bH^2/\L'(\Gamma)$. 
(See \cite{CDG2} for details.)
A subgroup of $\SOto$ is {\em Fuchsian} if it is a discrete subgroup without torsion. 
A Fuchsian group $\L'(\Gamma)$ is {\em convex cocompact} if 
one of the following holds: 
\begin{itemize}
\item  $\bH^{2}/\L'(\Gamma)$ has a compact convex core. 
\item $\bH^{2}/\L'(\Gamma)$ is geometrically finite without cusps
\item  
$\L'(\Gamma)$ contains no parabolic element.
\end{itemize} 

In this paper, we will be concerned with the cases without parabolic elements in $\Gamma$. 
The following conditions are equivalent conditions for a finitely generated orientation-preserving Lorentzian isometry group $\Gamma$. 
\begin{itemize} 
\item $\Lspace/\Gamma$ is a Margulis spacetime without parabolic holonomy elements.
\item An index-two subgroup $\Gamma \cap \Isom^{+}_{0}(\Lspace)$ or
$\Gamma$ itself is a proper affine deformation of a convex cocompact Fuchsian group.
(See \cite{GLM} for definition. ) 
\item $\Gamma$ is a nonamenable Lorentzian isometry group acting properly on $\Lspace$ without 
parabolic holonomy. 
\item $\Gamma$ is a free Lorentzian isometry group of rank $\geq 2$ acting properly and freely on $\Lspace$ without parabolic 
holonomy. 
\end{itemize} 
In this case, we call $\Lspace/\Gamma$ a {\em Margulis spacetime without cusp} for shortness.  

Notice that the definitions of hyperbolic, positive hyperbolic, and parabolic for elements of $\Gamma$ are 
the same as ones for ${\L'}(\Gamma)$ as a group of isometries of $\Ss_+$. 
(If a hyperbolic element in $\SOto$ has a negative eigenvalue, then the eigenvalues should be $-\lambda, -1/\lambda, 1$, 
where $\lambda> 0$.) 
It also follows from above that elements of $\Gamma$ are either hyperbolic or parabolic or identity. 
(See \cite{GLM}).

An {\em $\rpn$-structure} on a manifold is given by an atlas of charts to $\rpn$ with 
projective transition maps. 
Such geometric structures were first considered by Kuiper, 
Benz\'ecri, Vinberg, Koszul
and others in the 1950s and 1960s. Further developments 
can be followed in Goldman \cite{Gbook}, Choi and Goldman \cite{chgo}, Benoist \cite{Ben1}, 
and Cooper, Long, and Thistlethwaite \cite{Cooper2006} and \cite{CLT} and many others. 



\begin{thm}\label{thm:A}
Let $\Gamma$ be a free orientation-preserving Lorentzian isometry group of rank $\geq 2$ acting on   
on $\Ss_+\cup \Ss_-$ properly 
and freely
but without parabolic.
Then $\Gamma$ acts on an open domain 
${\mathcal D}\subset\SI^2_{\infty}$ 
such that ${\mathcal D}/\Gamma$ is a closed surface $\Sigma$.
Such a domain is unique up to the antipodal map $\mathscr A$.
\end{thm} 

These surfaces have the $\rptwo$-structures on 
closed surfaces of genus $g$, $g \geq 2$, discovered by Goldman~\cite{Gf}
in the late 1970s.  
Although their developing maps from the universal cover $\tilde \Sigma$ of some such surface $\Sigma$ 
are not covering-spaces onto their images in $\rptwo$, 
when lifted to the double-covering space $\SI^2_{\infty}$, 
the developing maps (remarkably)
are covering-spaces onto open domains. The surface is a quotient of 
a domain in $\SI^2_{\infty}$ by a group of projective automorphisms,
an $\rptwo$-analog of the standard Schottky uniformization
of a Riemann surface as a $\cpo$-manifold as observed by Goldman. 

An $n$-dimensional open manifold is said to be 
{\em tame} if it is homeomorphic to the interior of a compact $n$-manifold 
with boundary. 
A {\em handlebody} is a $3$-dimensional manifold obtained 
from a $3$-ball $B^3$ by attaching $1$-handles to $3$-dimensional balls.
The topology of handlebodies and related objects form a central topic in 
the $3$-dimensional manifold topology. See Hempel \cite{Hempel} for a thorough survey of 
this field of mathematics developed in the 20th century.

\begin{thm}\label{thm:B}
Let $M$ be the Margulis spacetime $\Lspace/\Gamma$ without cusp.
Then $M$ is homeomorphic to the interior of an orientable handlebody. 
\end{thm}


The proof is given in \S \ref{sec:tameness} by compactifying the Margulis spacetime into
a compact $\R\P^3$-manifold with $\R\P^2$-surface boundary. 
We remark that Frances \cite{CF} earlier produced conformal compactifications of
Margulis spacetimes, which are not manifolds as opposed to our projective compactifications. 
We are currently trying to generalize these techniques  to Lorentz manifolds with parabolics (see \cite{ChDG}) and  
to higher dimensional geometric structures and so on.


A complete geodesic on a quotient space of 
a hyperbolic space or a Lorentzian space  is {\em nonwandering}  if 
it is bounded in both directions, or equivalently, the closure of the forward part is compact 
and so is that of the backward part, or the $\alpha$-limit and the $\omega$-limit are both nonempty and compact.
This is the same as 
the term ``recurrent'' in \cite{GL} and \cite{CGJ}, which 
does not agree with the common usage in the dynamical system theory
whereas the recurrence set they discuss is the same as the nonwandering set as 
in Eberlein \cite{Eb} and Katok and Hasselblatt \cite{Katok2}. 
Both endpoints of the geodesics lie in the limit sets
and they comprise the nonwandering set 
by Corollary 3.8 in \cite{Eb}.






In \S\ref{sec:ProjGeo}, we review facts on hyperbolic surfaces, 
Fuchsian group actions, 
$\rptwo$-structures  and geometry that we need.
We compactify $\Lspace$ by a real projective $2$-sphere $\SI^2_{\infty}$ that is 
just the space of directions in $\Lspace$. 

In \S \ref{sec:Lspace}, we review facts about the Lorentzian space $\Lspace$ and affine boosts. 
The affine boosts will be extended to projective boosts in $\SI^3$. 

In \S \ref{sec:hyps}, we review $\R\P^2$-structures on surfaces and hyperbolic surfaces. 
We also discuss the dynamics of the geodesic flows on the unit tangent bundles of hyperbolic surfaces.

In \S \ref{sec:bord}, we prove Theorem \ref{thm:A}.
We find an open domain $\tSigma$ in 
the boundary $\SI^2_{\infty}$ of $\Lspace$ where $\Gamma$ acts properly and freely. 
$\tSigma/\Gamma$ is shown 
to be a closed $\rptwo$-surface $\Sigma$.
Such a surface was constructed by Goldman \cite{Gf}; however, 
we realize it as the quotient of an open domain in the sphere $\SI^2_{\infty}$, 
which is a union of two disks and infinitely many strips joining the two.
Here our choice of $\tilde \Sigma$ depends on the orientation.   

In \S \ref{sec:propdisc}, 
we recall the work of Goldman, Labourie, and Margulis \cite{GLM}
on nonwandering  geodesics on the Margulis spacetime. Their work in fact shows that 
all nonwandering  spacelike geodesics in a Margulis spacetime are in a uniformly bounded part.

We fix an orientation on $\Lspace$ and assume that the Margulis invariants are all positive. 
Our choice of $\tilde \Sigma$ above depends on the orientations and the signs of Margulis invariants. 
(See the statements after Theorem \ref{thm:C}.) 
Next, we prove the proper discontinuity of the action of the group $\Gamma$
on $\Lspace \cup \tSigma$. 
We first show that for a fixed Lorentzian isometry
$g$, we can put every compact properly convex subset of $\Lspace \cup \tSigma$ 
into an $\eps$-neighborhood of the ``attracting'' segment of $g$ in $\SI^2_{\infty}$ by $g^i$ for sufficiently large $i$.
(This is the key model of our proof.)

We next show that for any compact set $K \subset \Lspace \cup \tSigma$,
we have only finitely many $g \in \Gamma$ where $g(K) \cap K \ne \emp$.
Suppose not. We find a sequence $\{g_i\}$ with $g_i(K) \cap K \ne \emp$. 
Using \cite{GLM}, 
we can find an infinite sequence $\{g_i\}$ that behaves almost like $\{g^i\}$ in the
dynamical sense as $i \ra \infty$ up to small changes of stable and unstable planes. 
In other words, we can find a uniformly bounded sequence of coordinates to normalize $g_i$ into 
fixed forms and these coordinates are uniformly convergent.

In \S \ref{sec:tameness}, we prove the tameness of $\Lspace/\Gamma$, i.e., Theorem \ref{thm:B},
using the classical $3$-manifold topology as developed by Dehn and others. 
We conclude with the compactification Theorem \ref{thm:C}, proving Theorem \ref{thm:B}.  





There is another independent solution to this question by 
Danciger, Gu\'eritaud, and Kassel \cite{DGK1}.
Their approach uses the deformation of constant curvature Lorentzian manifolds and yields 
other results such as the deformability of Margulis spacetimes to anti-de Sitter spaces. 
Their method is to identify $\R^{3}$ as the Lie algebra $sl(2, \R)$ with the adjoint action by $\SL(2, \R)$. 
Then they show the existence of foliation of $M$ by complete lines fibering over the corresponding hyperbolic surface.
Compared to their approach, our approach 
mostly concentrates on $\R\P^n$-structures and the compactification of the spacetime as an $\R\P^3$-manifold. 

Danciger, Gu\'eritaud, and Kassel \cite{DGK2} also announced in September 2013 in ICERM
the solution to the crooked plane conjecture of Drumm-Goldman \cite{DG1}
that the Margulis spacetimes without cusp decomposes into $3$-balls by unions of mutually disjoint crooked planes. 
This implies the tameness also. With Theorem \ref{thm:A}, we can obtain compactifications, i.e., Theorem \ref{thm:B} also. 
For rank two cases, the conjecture was verified by 
Charette-Drumm-Goldman \cite{CDG3}. Related to this conjecture, 
we have Proposition \ref{prop:genschottky}. 

This work began from the question of Goldman on
the $\rptwo$-bordification of the Margulis spacetime during the Newton Institute Workshop 
``Representations of Surface Groups and Higgs Bundles'' held in Oxford, March 14--18, 2011. 
The authors thank 
the Newton Institute and the Institut Henri Poincar\'e where parts of this work were carried out. 
Finally, we thank Bill Thurston without whose teaching we could not have accomplished many of 
the things in this paper. 




\section*{Notation and terminology}
Let $V$ denotes a (real) vector space. We denote the associated {\em projective space,\/}
comprising $1$-dimensional linear subspaces of $V$ by $\P(V)$.
{\em Projective $n$-space\/} $\rpn := \P(\R^{n+1})$ also identifies with the quotient of
$\R^{n+1} \setminus \{0\}$ by the multiplicative group $\R^*$ of nonzero real scalars.
The {\em sphere of directions\/} $\SI^n$ consists of the quotient of $\R^{n+1} -\{0\}$ by the subgroup
$\R^+$ of positive real scalars. If $v\in\R^{n+1}$ and $v\neq 0$, then denote its equivalence
class in $\rpn$ by $[v] = [v^0 : v^1 : \dots : v^n ]$ and its equivalence class in $\SI^n$ by
$\llrrparen{ v}= \llrrparen{ v^0 : v^1 : \dots : v^n }$.
The map
\begin{align*}   \SI^n &\longrightarrow   \rpn \\
 \llrrparen{  v }  &\longmapsto [v] \end{align*}
is a double covering.  The coordinates $v^0, \dots v^n$ are the homogeneous coordinates
of the corresponding points $[v]\in\rpn$ and $\llrrparen{ v}\in \SI^n$.


The {\em sign\/} of a real number $x$ is defined as:
\[
\o{sgn} (x) := \begin{cases}   +1 &\text{~if~} x > 0 \\
0 &\text{~if~} x = 0 \\
-1 &\text{~if~} x < 0. \end{cases} \]
If $G$ is a group acting on a space $X$, and $S\subset X$ is a subset, then denote its {\em stabilizer\/} by:
\[
\mathsf{Stab}(S) := \{ g\in G \mid   g(S) = S \}.\]

\section{Projective geometry 
}\label{sec:ProjGeo}

\subsection{Projective space} \label{sub:rpnstr}

The projective space $\rpn$ is given by  $\P(\R^{n+1})$. 
We have a homomorphism from the general linear group 
\[\GL(n+1, \R) \ra \PGL(n+1, \R) := \GL(n+1, \R)/\sim\] 
where two linear maps $L_1$ and $L_2$ are equivalent if $L_1 = s L_2$ for $s \in \R - \{0\}$.
A map of an open subset of $\rpn$ to another one of $\rpn$ is {\em projective} if it is 
a restriction of an element of $\PGL(n+1, \R)$.

More generally, a nonzero linear map $L: \R^{n+1} \ra \R^{n+1}$ becomes 
a {\em projective endomorphism } 
\[ \hat L: \P(\R^{n+1}) - \P(N) \ra \P(\R^{n+1})  \hbox{ defined by } \hat L ([\vec v]) = [ L(\vec v)], \vec v \in \R^{n+1},\]
where $N$ is the kernel of $L$ and $\P(N)$ is the image of $N-\{O\}$ in $\P(\R^{n+1})$.
 (See \S 5.1 of Benz\'ecri \cite{Benzecri}.)
The space of projective endomorphisms equals the projectivization 
$\P(\Mat_{n+1}(\R))$ of the linear space $\Mat_{n+1}(\R)$ of all linear maps $\R^{n+1} \ra \R^{n+1}$. 
The space $\P(\Mat_{n+1}(\R))$ forms a compactification of the group $\PGL(n+1, \R)$ as observed first by 
Kuiper \cite{Kuiper} and developed by Benz\'ecri \cite{Benzecri}.

The projective geometry is given 
by a pair $(\rpn, \PGL(n+1, \R))$. 
An {\em $\rpn$-structure} on a manifold $M$ is given by a maximal atlas of 
charts to $\rpn$ with projective transition maps. 
The manifold $M$ with an $\rpn$-structure is said to 
be an {\em $\rpn$-manifold}. A {\em projective map} for two 
$\rpn$-manifolds $M$ and $N$ is a map $f: M \ra N$ so that $\phi \circ f \circ \psi^{-1}$ is projective 
whenever it is defined where $\phi$ is a chart for $N$ and $\psi^{-1}$ is a local inverse of a chart for $M$.

An $\rpn$-structure on a manifold $M$ gives us
an immersion $\dev: \tilde M \ra \rpn$ that is equivariant with respect 
to a homomorphism 
\begin{align*} 
& h: \pi_1(M) \ra \PGL(n+1, \R): \hbox{ that is,}  \\
& \dev \circ \gamma = h(\gamma) \circ \dev \hbox{ for every } \gamma \in \Gamma
 \end{align*} 
where $\tilde M$ is a regular cover of $M$ and $\Gamma$ is the deck transformation group of $\tilde M \ra M$. 
Here, $\dev$ is called a {\em developing map} and $h$ is called the {\em holonomy homomorphism}. 
$(\dev, h)$ is only determined up to an action of $\PGL(n+1, \R)$
where 
\[g(\dev, h(\cdot)) = (g \circ \dev, g h(\cdot) g^{-1}) \hbox{ for } g \in \PGL(n+1, \R).\]
Conversely, such a pair $(\dev, h)$ will give us an $\rpn$-structure
since $\dev$ gives us charts that have projective transition maps.
(See \cite{cdcr1} for details.)




\subsection{The sphere of directions} \label{subsec:notation}

We can identify the sphere $\SI^n$ with $\Ss(\R^{n+1})$.
There exists a quotient map $q_1: \SI^n \ra \rpn$ defined by sending 
a unit vector to its equivalence class. This is a double-covering map,
and it induces an $\rpn$-structure on $\SI^n$.
The group $\Aut(\SI^n)$ of projective automorphisms is isomorphic 
to the group $\SL_\pm(n+1, \R)$ of linear maps of determinant equal to $\pm 1$
with the quotient homomorphism $\SL_\pm(n+1, \R) \ra \PGL(n+1, \R)$
with the kernel $\pm \Idd$. It will be convenient to think of the elements of $\Aut(\SI^n)$ as matrices
in $\SL_\pm(n+1, \R)$. 

This space is again an open dense subspace of $\Ss(\Mat_n(\R))$. 
We call each element of $\Ss(\Mat_n(\R))$ a {\em projective endomorphism} of $\SI^n$.
Each element corresponds to a map 
\[\hat L: \Ss(\R^{n+1}) - \Ss(N) \ra \Ss(\R^{n+1}) 
\hbox{ given by }  \hat L (\llrrparen{ \vec v }) = \llrrparen{ L (\vec v)}, \vec v \in \R^{n+1}\]
where $N$ is the kernel of a linear map $L$ in $\Mat_n(\R)$. 

For an open domain $\mathcal D$ in $\SI^n$ and a discrete group $\Gamma$ in $\Aut(\SI^n)$ 
acting on it properly 
and freely, ${\mathcal D}/\Gamma$ has 
an $\rpn$-structure since $q_1|{\mathcal D}$ gives an immersion to $\rpn$ and the homomorphism 
\[\Gamma \subset \Aut(\SI^n) \ra \PGL(n+1, \R)\] 
gives the associated holonomy homomorphism.

In this paper, we will study objects on $\SI^3 
= \Ss(\R \oplus \R^3)$. 



The antipodal map 
\begin{align*}
\SI^{n} &\xrightarrow{\A}  \SI^n \\
 \llrrparen{\vv} &\longmapsto \llrrparen{-\vv}
 \end{align*}
generates the deck transformation group $\SI^n \ra \rpn$. 
Given a subset or a point $K$ of $\SI^3$, we denote by $K_-$ the set of antipodal points of points of  $K$ or the antipodal point
respectively. 

Given a subset $A$ of $\SI^{n}$, 
$\o{span}(A)$ is the unique minimal projective subspace $S$ of $\SI^{n}$ containing $A$.

\subsubsection{Elliptic geometry}

A {\em subspace} of $\SI^3$ is a subset defined by a system of linear equations in $\R^4$. 
A {\em line} is a $1$-dimensional subspace.  A singleton is not a subspace. 
A {\em projective geodesic} is an arc in a line. 
We will be using the standard Riemannian metric $\bdd$ on $\SI^3$.
Notice that the geodesic structure of $\SI^3$ is the same as the projective one. 
Thus, we will use the same term ``geodesics'' for both types of geodesics. 
Geodesics and projective subspaces are all subsets of lines or subspaces in $\SI^3$. 
A projective automorphism $g$ sends these to themselves, while $g$ does not preserve any metric.
A pair of {\em antipodal points} on $\SI^3$ is the equivalence class $\llrrparen{ v}$ of a nonzero vector $v$ 
and one $\llrrparen{ -v}$ of $v_{-}$. 
A segment connecting two antipodal points is precisely a segment of $\bdd$-length $\pi$.
(For these topics, see 
\S 34 of Eisenhart \cite{Eisen}.)

Here, we denote by $\ovl{pqp_-}$ the unique closed segment in 
$\SI^3$ with endpoints $p$ and $p_-$ passing through $q$
not equal to $p$ or $p_-$. By $\ovl{pq}$, we denote the unique closed segment in $\SI^3$ with endpoints $p$ and $q$ 
not containing $p_-, q_-$ provided $p \ne q$ and $p \ne q_-$. 
The notation $\ovl{pq}^o$ or $\ovl{pqp_-}^o$ indicates the interior of the segment.


A {\em convex subset} of $\SI^{3}$ is a subset where every pair of points can be connected by 
a segment of $\bdd$-length $\leq \pi$. 
A  subset $A$ of $\SI^{3}$ is {\em properly convex} if $A$ is a precompact convex subset of some open hemisphere.
A compact convex set is properly convex if and only if there are no antipodal pair of points. 
These notions are projectively invariant. 
(See \cite{psconv} for details.)   

\subsubsection{Bi-Lipschitz maps}
A {\em bi-Lipschitz map} of $\SI^3$ is a homeomorphism $\SI^3 \xrightarrow{f} \SI^3$ so that 
\[C^{-1}\bdd(x, y) \leq \bdd(f(x), f(y)) \leq C \bdd(x, y) \hbox{ for } x, y \in \SI^3\]
where $C > 0$ is independent of $x, y$. 
Elements of $\Aut(\SI^3)$ including ones extending the Lorentzian isometries 
are bi-Lipschitz maps of $\SI^3$.

Finally, for the purpose of drawing figures only, we map $\Lspace$ to the unit $3$-ball in $\R^3$ by the {\em normalization map} 
\begin{equation}\label{eqn:normalization}
 \llrrparen{ x: y: z: 1 } \ra  \frac{(x, y, z)}{\sqrt{x^2 + y^2 + z^2+1}}. 
 \end{equation} 
This will identify $\SI^2_{\infty}$ with the sphere of radius $1$ in $\R^3$ 
and as a codimension-one projective subspace of $\SI^3$, represented stereographically. 





\subsection{A matrix lemma} 
We denote 
\begin{align} 
& \be_{1}= \llrrparen{ \vve_1}, \be_{2}=\llrrparen{ \vve_2}, \be_{3}=\llrrparen{ \vve_3} , \be_{4}=\llrrparen{ \vve_4}, \nonumber \\
&  \be_{1-}=\llrrparen{ -\vve_1}, \be_{2-}=\llrrparen{ -\vve_2 }, \be_{3-}=\llrrparen{ -\vve_3},  \be_{4-}=\llrrparen{ -\vve_4}. \nonumber 
\end{align}
There exists a standard metric on $\R^4$ where these vectors form an orthonormal basis.
We will use the norm of this metric to be called the {\em Euclidean norm}. 
The {\em norm} of a matrix $A$ is $\max\{|a_{ij}|\}_{i=1, \dots, 4, j=1, \dots, 4}$ for entries $a_{ij}$ of $A$.

\begin{lem}[Uniform convergence]\label{lem:matrixcor} 
Let $\vv_i^j$ for $j =1,2,3,4$ be a sequence of points of $\SI^3$. 
Suppose that $\vv_i^j \ra \vv_\infty^j$ for each $j$ 
and mutually distinct points $\vv_\infty^1, \dots, \vv_\infty^4$
formed by independent vectors in $\R^4$.
Then there exists an integer $I_{0}$ 
and a sequence $h_i$ for $i > I_{0}$ of elements of $\Aut(\SI^3)$ so that 
\begin{itemize}
\item $h_i(\vv_i^j) = \be_j$,
\item  $\{h_i\} \ra h_{\infty}$ and $\{h_i^{-1}\} \ra h_{\infty}^{-1}$ 
uniformly under the $C^s$-topology for every $s \geq 0$.
\end{itemize}  
\end{lem}
\begin{proof} 
Let $\vec{v}_i^j$ represent the unit Euclidean norm vector corresponding to $\vv_i^j$. 
We find $\hat h_i \in \Aut(\SI^3)$ so that 
$\hat h_i(\be_j) = \vv_i^j$ for each $i$. 
We set $M_i(\vec{e_i}) = \vec{v}_i^j$; 
that is, the column vectors of $M_i$ equals $\vec{v}_i^{\,1}, \vec{v}_i^{\, 2}, \vec{v}_i^{\, 3}, \vec{v}_i^{\, 4}$. 
Then clearly, the sequence $\{M_i\}$ converges to a $4\times 4$-matrix $M_\infty$ of nonzero determinant. 
Therefore, there exists an integer $I_{0}$ where the sequence $\{M_i^{-1}\}$ for $i > I_{0}$ of inverse matrices 
converges to $M_\infty^{-1}$.

We let $h_i$ be the projective automorphism in $\Aut(\SI^3)$ induced by $M_i^{-1}$, 
let $h_\infty$ be the one corresponding to $M_\infty^{-1}$, and this satisfies
the properties we need. 
\end{proof}

\subsection{The space of properly convex sets and joins} \label{subsec:space}

Let $\bdd\hbox{-diam}(A)$ denote the diameter of the set; that is, the supremum of the set of 
distances between every pair of points of a subset $A$ of $\SI^n$. 
Given a pair of points or subsets $A$ and $B$ in $\SI^n$, we define the infimum distance
\[\bdd(A, B) = \inf \{\bdd(x, y)| x \in A, y \in B\}.\]
If $\bdd(A, B) > \eps$ for $\eps > 0$, then we say that $A$ is {\em bounded away} from $B$ by $\eps$.  

An {\em $\eps$-neighborhood} $N_\eps(A)$ for a number $\eps > 0$ of a subset or a point $A$ of $\SI^3$ 
is the set of points of $\bdd$-distances less than $\eps$ from some points of $A$.
We define the geometric Hausdorff distance 
$\bdd^H(A, B)$ between two compact subsets $A$ and $B$ of $\SI^3$ 
to be 
\[\inf \{\eps > 0| B \subset N_\eps(A) \hbox{ and } A \subset N_\eps(B)\}. \]
A sequence of compact sets $\{K_n\}$ in $\SI^3$ 
{\em converges} (or {\em converges geometrically}) to a compact 
subset $K$ if $\bdd^{H}(K_{n}, K) \ra 0$. 
In other words, for every $\eps > 0$, there exists $N$ so that  
\[ K \subset N_\eps(K_n) \hbox{ and } K_n \subset N_\eps(K) \hbox{ for } n > N. \]
We will simply write $\{K_n\} \ra K$. 

The following are commonly known facts with elementary proofs (see pages 280--281 of Munkres \cite{Munkres}). 
\begin{itemize}   
\item The space of compact subsets of a compact Hausdorff space with this metric is compact Hausdorff also 
and hence every sequence has a convergent subsequence.
\item given a convergent sequence $\{K_n\}$ of compact subsets of $\SI^3$, 
if a sequence $\{J_n\}$ is such that $J_n \subset K_n$, then 
any geometric limit of a subsequence of $\{J_n\}$ is a subset of the geometric limit of $\{K_n\}$.
\item given sequences $\{K_n\}$ and $\{\hat K_{n}\}$ of compact subsets of $\SI^3$ converging to $K_{\infty}$, 
if a sequence $\{J_n\}$ is such that $K_{n}\subset J_n \subset \hat K_n$ for every $n$, then 
$\{J_{n}\} \ra K_{\infty}$.  
\item If we have $\{K^j_n\} \ra K^j$ as $n \ra \infty$ for each $j=1, \dots, m_0$, then
\[\bigcup_{j = 1}^{m_0} K^j_n \ra \bigcup_{j=1}^{m_0} K^j\] holds as $n \ra \infty$.
\end{itemize} 
Note that 
\[\bdd(A, B) \leq \bdd^H(A, B) .\]

\subsubsection{Joins}
Suppose that two compact subsets $A$ and $B$ of $\SI^3$ have no pairs of points $x \in A, y \in B$ 
so that $x$ and $y$ are antipodal, that is, $A \cap B_{-} = \emp$. 
Then the {\em join} $\Join(A,B)$ is the union of all segments of $\bdd$-lengths $< \pi$
with single endpoints in $A$ and the other ones in $B$. In set theoretic terms, we have
\newcommand{\bfa}{\mathbf{a}}
\newcommand{\bfb}{\mathbf{b}}
\begin{equation} \label{eqn:join}
\Join(A, B) = 
\Big\{ \llrrparen{ c_1\bfa + c_2 \bfb } \; \Big| 
\;  \llrrparen{\bfa } \in A, \  \llrrparen{\bfb } \in B, c_1,c_2 \geq 0  \Big\}.
\end{equation}


If $A$ and $B$ are convex and compact and $A \cup B$ has no pair of antipodal points, i.e., 
$A$ and $B$ are convex and compact and $A \cap B_{-} = \emp$, 
then $\Join(A,B)$ is defined 
and also convex. (See Proposition 2.4 of \cite{psconv}.) 
For any projective automorphism $g \in \Aut(\SI^3)$, 
\[ g \big(\Join(A, B)\big) = \Join\big(g(A), g(B)\big). \]

If $A$ and $B$ are compact and have no antipodal pair of a point of $A$ and a point $B$, 
then $\Join(A, B)$ is well-defined and is also compact as we can easily show that $\Join(A, B)$ is closed. 

For example, when $A$ is a great circle in $\SI^3$ and $b\not\in A$, 
every segment from $b$ to $A$ has  $\bdd$-length $< \pi-\eps$ for a fixed $\eps>0$.
In this case, $\Join(A,B)$ is a {\em $2$-dimensional hemisphere}, a disk in a hyperspace with geodesic boundary in $\SI^3$, 
a convex set.


\begin{lem}\label{lem:geoconv} 
Let $A_n$ and $B_n$ be a sequence of properly convex compact subsets of $\SI^3$ geometrically converging to 
properly convex compact sets $A$ and $B$ respectively. Suppose that 
$A \cap B_{-} = \emp$. Then $J_n = \Join(A_n, B_n)$ is defined for $n > I_0$ for 
sufficiently large $I_0$ and the sequence $\{J_n\}_{n > I_0}$ converges to the join $J$ of $A$ and $B$. 
\end{lem}
\begin{proof} 
Choose $\eps> 0$ so that $N_\eps(A)$ and $N_\eps(B)$ are properly convex and 
$N_\eps(A) \cap (N_\eps(B))_{-} =\emp$. 
Choose $I_0$ so that $A_n \subset N_\eps(A), B_n \subset N_\eps(B)$ for $n > I_0$. 
Clearly, $A_n$ and $B_n$ have no antipodal pair of points for $n > I_0$. 
One can parameterize each geodesic segment connecting a point of $A_n$ to that of $B_n$ 
by arclength in $\bdd$. From this, the lemma follows. 
\end{proof}



\section{Affine $3$-dimensional Lorentzian geometry} \label{sec:Lspace}

\subsection{The $3$-dimensional Minkowski space}  
The Lorentzian space $\Lspace$ is the affine space $\R^3$ equipped 
with a nondegenerate bilinear form $\cdot$
of signature $1,1, -1$. If we choose the origin in $\Lspace$, we obtain 
the Lorentzian vector space $\V$. 

One identifies $\R^3$ with  
$\Lspace$ and the complement of a subspace of codimension one 
in $\rpthree$, so called the complete affine subspace.
$\R^3 = \Lspace$ lifts to a subspace of $\SI^3$ double-covering $\rpthree$. 
The closure of the lifted $\Lspace$ is a standard $3$-dimensional hemisphere $\He$,
and $\Lspace$ identifies with the open hemisphere $\He^o$.
The boundary $\SI^2_{\infty}$ of $\Lspace$ is the subspace of codimension one 
identifiable with the $2$-sphere of directions in $\R^3$.

The sphere $\SI^2_{\infty}$ corresponds to the hyperplane given by $t = 0$
and the origin of $\Lspace$ is given by $\llrrparen{ 0: 0: 0: 1}$\,, denoted by $O$. 
For this system of coordinates, the point $(x, y, z) \in \Lspace$ is given 
coordinates $\llrrparen{ x: y: z: 1}$. 

$\SI^2_{\infty}$ also has a homogeneous coordinate system $\llrrparen{ x: y: z}$ assigned 
to a vector $p$ in $\SI^2_{\infty}$ of unit length if $\frac{(x,  y, z)}{||(x, y, z)||} = p$. 
In the larger coordinate system $\llrrparen{ x: y: z}$ is identical with $\llrrparen{ x: y: z: 0}.$ 

We will also be using a fixed Euclidean metric $d_E$ on $\Lspace$ compatible with 
the affine coordinate system. It is of great importance that projective geodesics, spherical geodesics, and 
Euclidean geodesics on $\Lspace$ are the same ones up to parametrizations.  

To summarize, the subsets of $\SI^3$ identified are:  
\begin{align*} \label{eqn:ident} 
\SI^3 &= 
 \Big\{ \llrrparen{ x: y: z: t}\; \Big|\;  x, y, z, t \in \R, (x, y, z, t) \ne (0, 0, 0,0)\Big\},\\
\He  &= 
  \Big\{ \llrrparen{ x: y: z: t} \;\Big|\;  x, y, z,  t \in \R, t \geq 0, (x, y, z, t) \ne (0, 0, 0,0) \Big\}, \\
\R^3 & = \Lspace = \He^o = \Big\{ \llrrparen{ x: y: z: t} \;\Big|\; x, y, z, t \in R, t>0\Big\}, \\
&= \Big\{ \llrrparen{ x: y: z: 1} \Big| x, y, z \in \R\Big\}, \\
\SI^2_{\infty} & =  
 \Big\{\llrrparen{ x: y: z: 0} \;\Big|\; x, y, z \in \R, (x, y, z, 0) \ne (0, 0, 0, 0) \Big\}.
\end{align*}

%

\subsection{Lorentz isometries}
An orientation-preserving Lorentz isometry is represented as an affine transformation
\begin{align*} 
\Lspace &\xrightarrow{\gamma} \Lspace \\
\bx &\mapsto A  \bx +  \bb \end{align*}
for $\bx \in \Lspace$.

Such an element can be represented by a matrix
\[
\left(
\begin{array}{cc}
  A  &  \bb\\
  0   &  1 \\
\end{array}
\right) \in \GL(4, \R)
\]
with a nonsingular $3\times 3$-matrix $A$ in $\SO(2, 1)$ and $\bb\in\R^3$. 

The group $\Isom^+(\Lspace)$ of orientation-preserving Lorentzian transformations of $\Lspace$
is a subgroup of the group $\Aff^+(\Lspace)$ of orientation-preserving affine transformations of $\Lspace$. 
$\Aff^+(\Lspace)$ identifies with the subgroup of $\Aut(\SI^3)$ of 
orientation-preserving elements acting on $\R^3=\Lspace$,
preserving $\SI^2_{\infty}$, since any affine transformation of $\R^3$ extends to a projective automorphism 
of $\R\P^3$ and hence to an automorphism of $\SI^3$ analytically. 
We obtain embeddings
\[ \Isom^+(\Lspace) \hookrightarrow \Aff^+(\Lspace) \hookrightarrow \Aut(\SI^3). \]
The elements of respective groups are called {\em projective automorphisms}
of {\em Lorentzian} type and ones of {\em affine} type or just {\em Lorentzian isometries} or 
{\em affine transformations}. 

\subsection{Oriented Lorentzian vector spaces} \label{sub:orlo}
Let $(\V, \cdot,\Det)$  
denote a three-dimensional {\em oriented Lorentzian\/}
$\R$-vector space. That is, $\V\cong\R^3$ and is given a symmetric bilinear
form
\[\V \times \V \rightarrow  \R \hbox{ given by }  (\vv,\vu) \mapsto \vv \cdot \vu\] 
of index $1$, and a nondegenerate alternating trilinear form
\begin{align*}
 \V \times \V \times \V &\longrightarrow \R \\
(\vv,\vu, \vw) &\longmapsto \Det(\vv,\vu,\vw).\end{align*} 
The automorphism group of $(\V, \cdot, \Det)$ 
is the special orthogonal group $\SOto$.



The Lorentzian structure divides $\SI^2_{\infty}$ into three open domains 
$\Ss_+, \Ss_0,$ and $\Ss_-$, 
which are separated by two conics $\partial\Ss_+$ and $\partial\Ss_-$ corresponding
to the nullcone in $\V$.  
Let $\Ss_+\subset\Ss$ denote the set of future timelike directions; 
its boundary $\partial\Ss_+$ consists of future null directions. 
Similarly, let $\Ss_-\subset\Ss$ denote the set of past timelike directions with boundary 
$\partial\Ss_-$ consisting of past null directions. 
Both $\Ss_+$ and $\Ss_-$ are cells, interchanged by $\A$.  
The set $\Ss_0$ of spacelike directions forms an  $\A$-invariant annulus
bounded by
the disjoint union $\partial\Ss_+ \coprod \partial\Ss_-$.

Recall that $\Ss_+$ is the Beltrami-Klein model of the hyperbolic plane
where $\SOto^o$ acts as the orientation-preserving isometry group. 
Here the metric geodesics are precisely the projective geodesics and vice versa.  
(We will use the same term ``geodesics'' for both types of geodesics for ones in $\Ss_+$ and 
later $\Ss_-$.) 
The geodesics in $\Ss_+$ are straight arcs and $\partial\Ss_+$ forms 
the ideal boundary of $\Ss_+$. For a finitely generated discrete subgroup $\Gamma$ in 
$\SOto^o$, the surface $\Ss_+/\Gamma$ has a complete hyperbolic structure  
as well as an $\rptwo$-structure with the compatible geodesic structure. 

Alternatively, we can use $\Ss_-$ as the model and 
$\Ss_-/\Gamma$ has a complete hyperbolic structure as well as
an $\rptwo$-structure.
Since $\Ss_+$ and $\Ss_-$ are properly convex, these are examples of 
properly convex $\rptwo$-surfaces.

\subsection{Null half-planes}\label{sub:nullh}
Let $\mathcal N$ denote the {\em nullcone\/} in $\V$, that is, the collection 
of all vectors $\vv\in\V$ with $\vv\cdot\vv = 0$. 
Its projectivization $\P({\mathcal N} - \{O\})$ consists of all {\em null lines\/} in $\V$. 
Suppose $\vv\in {\mathcal N} - \{O\}$.
Its orthogonal complement  $\vv^\perp$ is a 
{\em null plane\/} which contains the line $\R\vv$. The line $\R\vv$
separates $\vv^\perp$ into two half-planes. 

The orientation on $\V$ determines an $\SOto$-invariant way to uniquely associate a component of 
$\vv^\perp - \R\vv$ to $\R\vv$ as follows.
Since $\vv\in {\mathcal N}$ holds, its direction $\llrrparen{ \vv}$ lies in either $\partial\Ss_+$ or 
$\partial\Ss_-$. Choose an arbitrary element $\vu$ of $\Ss_+$ or $\Ss_-$ 
so that the directions of $\vv$ and $\vu$ both lie in 
$\clo(\Ss_+)$ or $\clo(\Ss_-)$ respectively.
(For example $\vu = (0, 0, \pm 1)$ would be sufficient.)
Define the {\em null half-plane $\W(\vv)$}
(or the {\em wing\/}) associated to $\vv$\/ as:
$$
\W(\vv) := \left\{ \vw\in\vv^\perp \mid \Det(\vv,\vw,\vu) > 0\right\} \subset \vv^\perp - \R\vv.
$$
Since the Lorentzian product $\vu \times \vv$ is in $\W(\vv)$, 
the directions in the wing satisfy the right-hand rule. 
Since $\W(\vv) = \W\big(\vv_{-})$ holds, 
the null half-plane $\W(\vv)$ depends only on $\R\vv$.
The corresponding set of directions is the open arc 
\[ \varepsilon (\llrrparen{\vv}) := \llrrparen{\W(\vv)} \] 
in $\Ss_0$ joining $\llrrparen{\vv}$ to  
its antipode  $\llrrparen{\vv_{-}}$. 
Since $\vv^\perp$ is tangent to $\mathcal N$, the arc $\varepsilon\big(\llrrparen{\vv}\big)$ is tangent to  $\partial \Ss_+$. 
The orientation of $\partial\Ss_+$ induced from $\Ss_+$
agrees with the orientation of $\varepsilon(\llrrparen{\vv})$ away from $\llrrparen{\vv}.$

The corresponding map 
\[
\llrrparen{\vv}\longmapsto \varepsilon\big(\llrrparen{\vv}\big) \] 
is an $\SOto$-equivariant map 
\[ \partial\Ss_+ \ra \mathcal{S}\] 
where $\mathcal S$ denotes the set of half-arcs of form 
$\varepsilon\big(\llrrparen{\vv}\big)$ for $\llrrparen{ \vv } \in \partial\Ss_+$. 
The arcs $\varepsilon\big(\llrrparen{ \vv}\big)$ for $\llrrparen{ \vv } \in \partial\Ss_+$ foliate $\Ss_0$. 
We can obtain all arcs $\varepsilon\big(\llrrparen{\vv}\big)$ from 
$\varepsilon(\llrrparen{ 0:1:1})$ by an $\SI^1$-action fixing $\llrrparen{ \pm (0: 0: 1)}$
for a subgroup $\SI^1$ of $\SOto$.
Let us call the foliation $\mathcal F$. 

Hence $\Ss_0$ has an $\SOto$-equivariant 
quotient map
\[ 
\Ss_0 \xrightarrow{\Pi}   \P({\mathcal N} - \{O\}) \cong \SI^1
\]
and $\varepsilon\big(\llrrparen{\vv}\big) = \Pi^{-1}([\vv])$ for each $\vv \in {\mathcal N} -\{O\}$.

Consider a future-pointing null vector of Euclidean length $\sqrt{2}$:
\[
\vn_\theta := \left[\begin{array}{ccc} \cos(\theta) \\ \sin(\theta) \\ 1 \end{array}\right] \hbox{ for } \theta \in [0, 2\pi).
\]
Then $\W(\vn_\theta)$ consists of  all
\[\vp_\theta(t,s) := t \vn_\theta + s \left[\begin{array}{c} -\sin(\theta) \\ \cos(\theta) \\ 0 \end{array}\right]
\hbox{ for } t \in\R, s > 0.\] 
The arc $c_\theta := \varepsilon\big(\llrrparen{\vn_\theta} \big)$
on the sphere of directions $\SI^2_{\infty}$ is parametrized by unit vectors
\[\frac1{\sqrt{2}} \vp_\theta(t, \sqrt{1-t^2}) \hbox{ as } -1< t < 1.\]
(As an element of $\SI^3$, its homogeneous coordinates equal 
\[\llrrparen{ \frac1{\sqrt{2}} \vp_\theta(t, \sqrt{1-t^2}): 0} \]
where $\frac1{\sqrt{2}} \vp_\theta(t, \sqrt{1-t^2})$ is used as the first three coordinates.)


\begin{figure}[ht]
\floatbox[{\capbeside\thisfloatsetup{capbesideposition={right,top},capbesidewidth=5.5cm}}]{figure}[\FBwidth]
{\caption{{\small The tangent geodesics to disks $\Ss_+$ and $\Ss_-$  in the unit sphere $\SI^2_{\infty}$ embedded in $\R^3$ as mapped 
by the normalization map of \eqref{eqn:normalization}.}}\label{fig:3dfig}}
{\includegraphics
[width=8cm]{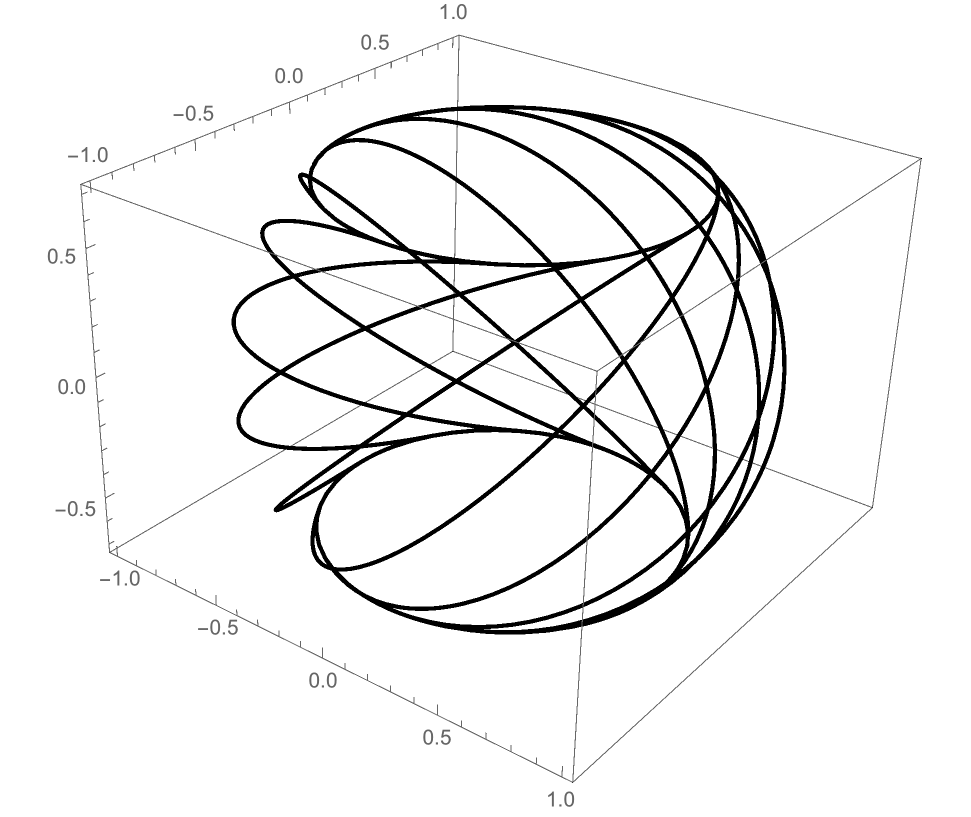}}
\end{figure}

\begin{figure}[htb]
\floatbox[{\capbeside\thisfloatsetup{capbesideposition={right,top},capbesidewidth=5.5cm}}]{figure}[\FBwidth]
{\caption{ {\small The tangent geodesics to disks $\Ss_+$ and $\Ss_-$ in the stereographically projected $\SI^2_{\infty}$ from $(0, 0, -1)$.
Radial arcs are geodesics. The inner circle represents the boundary of $\Ss_+$. The arcs of form $\varepsilon(x)$ for $x \in \partial\Ss_+$
are leaves of the foliation $\mathcal F$ on $\Ss_0$.}} \label{fig:figure2}}
{\includegraphics
[width=8cm]{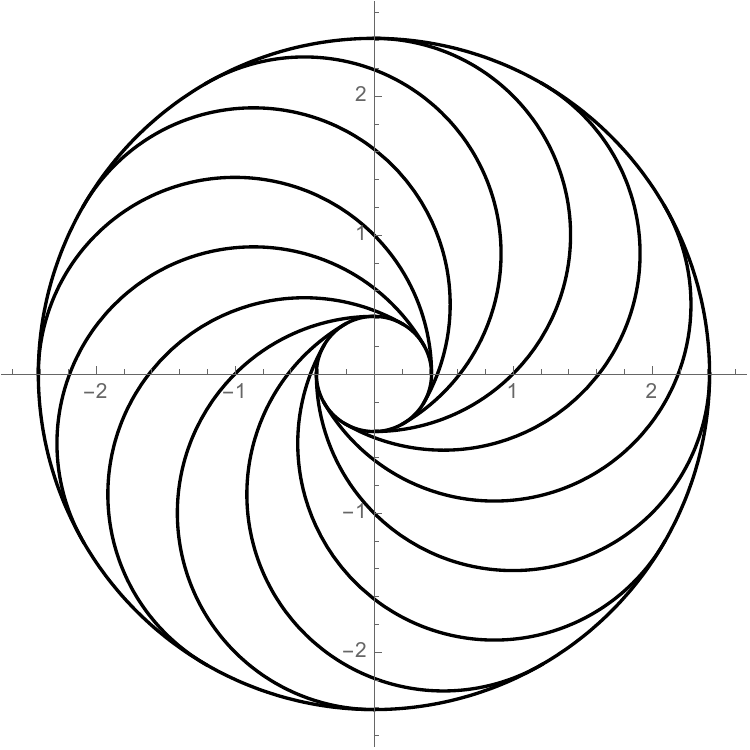}}
\end{figure}



The other component of 
$\vv^\perp - \R\vv$ could also be used as a wing
$\W(\vv)$, but for the opposite orientation on $\V$. 
Alternatively, these {\em negatively oriented null half-planes\/}
are the images of the positively oriented null half-planes
under the antipodal map $\A$ (which reverses orientation in
dimension three).  This phenomenon appears in the theory
of crooked planes, where one must choose a class of 
orientations for crooked planes, as well as the Margulis invariant,
where the choice of orientation plays a crucial role.

Since we have no need to consider negatively oriented null half-planes,
we henceforth restrict our attention to positively oriented null half-planes
in this paper.


\subsection{Affine boosts}\label{sec:AffineBoost}
\begin{defn}
An isometry of $\Lspace$ is an ({\em affine}) {\em boost\/} if its linear part is a positive hyperbolic element of $\SOto$ and it acts freely 
on $\Lspace$.
\end{defn}

Suppose that $\gamma\in\Isom^+(\Lspace)$ is an affine boost.
Then $\gamma$ preserves a unique line in $\Lspace$ which we denote
$\Axis(\gamma)$. 
Furthermore $\Axis(\gamma)$ is spacelike and 
the restriction of $\gamma$ to $\Axis(\gamma)$ 
is a nonzero translation. In a suitable coordinate system,
we can take $\Axis(\gamma)$ to be the $y$-axis, 
in which case $\gamma$ is the affine transformation:
\begin{equation} \label{eq:AffineBoost}
p \stackrel{\gamma}\longmapsto
\bmatrix e^l & 0 & 0 \\ 0 & 1 & 0 \\ 0 & 0 & e^{-l} \endbmatrix
p +
\bmatrix 0 \\  \alpha \\ 0\endbmatrix 
\end{equation}
where
\[
\bmatrix e^l & 0 & 0 \\ 0 & 1 & 0 \\ 0 & 0 & e^{-l} \endbmatrix
\]
is the linear part ${\L}(\gamma)$ of $\gamma$
and $p \in \Lspace$.
We say that $\gamma$ is in {\em standard form\/} if $\Axis(\gamma)$ is the $y$-axis.
In that case $\gamma$ is given by \eqref{eq:AffineBoost} as above.

\subsubsection{Oriented axes}
As observed by Margulis~\cite{Mar}, the axis of an affine boost $\gamma$ admits
a {\em canonical\/} orientation, induced by the ambient orientation of $\Lspace$.

Let $g = {\L}(\gamma)\in\SOto$ be positive hyperbolic. 
Then since $g$ is an isometry, and unimodular, the eigenvalues of $g$ are 
$\lambda,1, \lambda^{-1}$ where $\lambda  > 1$. 
The eigenspaces for eigenvalues $\lambda,\lambda^{-1}$ are null
(since $g$ is an isometry); choose respective eigenvectors $\vv_+(g), \vv_-(g) \in\V$
with the same causal character (that is, they are either both future-pointing
or both past-pointing). 
The $1$-eigenspace $\vv_0(g)\subset \V$ is spacelike and orthogonal to 
$\vv_+(g)$ and $\vv_-(g)$. We define $\vv_+(\gamma):= \vv_{+}(g), \vv_0(\gamma):=\vv_{0}(g), \vv_-(\gamma):=\vv_{-}(g)$. 

\begin{defn} \label{defn:neutvec} 
The {\em neutral eigenvector\/} of $\gamma$ is the unique unit-spacelike eigenvector 
$\vv_0(\L(\gamma))$ such that $\{\vv_+(\gamma), \vv_0(\gamma), \vv_-(\gamma)\}$ is a positively oriented basis
of $\V$. (See \cite{DG2} for details.)
\end{defn}

Since $\Axis(\gamma)$ is parallel to $\vv_0\big({\L}(\gamma)\big)$,
the neutral eigenvector $\vv_0(\gamma)$ of ${\L}(\gamma)$ defines an orientation on $\Axis(\gamma)$.
Denote the corresponding {\em oriented geodesic\/} by $\hAxis(\gamma)$.
By \eqref{eq:AffineBoost}, the restriction of $\gamma$ to $\Axis(\gamma)$ is a translation.
Since $\Axis(\gamma)$ is spacelike, the {\em displacement\/} of $\gamma$ along $\hAxis(\gamma)$ is
a well-defined real number
\[\mu(\gamma):= (\gamma(p) - p) \cdot \vv_0(\gamma) = \alpha \in \R,  p \in \Lspace,\]
called the {\em Margulis invariant.}  
For the basic properties of this invariant, see  \cite{Mar}, \cite{GLM}, \cite{Abels}, \cite{CharetteDrummIso}, 
\cite{CDG}.

\subsubsection{Invariant planes}
The weak-stable plane $W^{ws}(\gamma)\subset\Lspace$ 
is the affine subspace containing $\Axis(\gamma)$ and parallel to  the plane spanned by $\vv_0(\gamma)$ and $\vv_-(\gamma)$.  
Points in this affine plane asymptotically approach $\Axis(\gamma)$:
If $p\in W^{ws}(\gamma)$, then 
\begin{align*} 
\bdd\big(\, \gamma^n(p), \llrrparen{ \vv_0(\gamma):0}\big) &\ra 0, \\
\bdd\big(\,\gamma^n(p),\,  \Axis(\gamma)\, \big) & \ra 0 
\end{align*}
as $ n\ra \infty$.
\subsubsection{Classification of affine boosts}
An affine boost $\gamma$ is completely determined by the oriented spacelike line $\widehat\Axis(\gamma)$ and
the two real numbers $\ell > 0$ and $\alpha\neq 0$. 
Its conjugacy class is determined by the pair $(\ell,\alpha)\in\R^2$. 


\subsection{Projective boosts} \label{subsec:pboost}

\begin{defn}
A {\em projective boost\/} is a collineation of $\SI^3$ 
which extends an affine boost $g$ on $\Lspace$. 
Namely, an affine boost $\Lspace\xrightarrow{g}\Lspace$ extends
to a projective transformation of $\mathbb{P}^3$ and lifts to a
collineation $\gamma$  of the double covering $\SI^3$.
\end{defn}

\subsubsection{Standard form of a projective boost}\label{sec:StandardBoostForm}
A {\em standard projective boost\/} is the extension of
an affine boost in standard form to $\SI^3$.
Explicitly, $\gamma$ is given by a $4\times 4$-matrix  in a suitable coordinate system 
\begin{equation} 
\label{eq:StandardProjectiveBoostForm} 
\bmatrix 
e^\ell & 0 & 0 & 0 \\ 0 & 1 & 0 & \alpha  \\ 0 & 0 & e^{-\ell} & 0 \\ 0 & 0 & 0 & 1
\endbmatrix 
\end{equation} 
where $\ell > 0$ 
and $\alpha \neq 0$, which will be related to the Margulis invariant.  
It restricts to the affine boost
defined by \eqref{eq:AffineBoost} on the affine patch 
\[ \bmatrix x \\ y \\ z \endbmatrix \longmapsto \llrrparen{  x : y: z : 1}. \] 
The six fixed points on $\SI^2_{\infty}$ are:  
\begin{align*}
x^+_\pm &:= \llrrparen{ \pm 1:0:0:0}, \\
x^0_\pm &:= \llrrparen{ 0:\pm 1:0:0}, \\
x^-_\pm &:= \llrrparen{ 0:0: \pm 1:0}. 
\end{align*}
in homogeneous coordinates on $\SI^2_{\infty}$. 
For $l \ne 0, \alpha \ne 0$, these are the only fixed points on $\SI^3$. 
Also, $\gamma$ acts on the subspace $\sigma := [0:*:0:*]$. 

Such a collineation induces the affine isometry of $\Lspace$ defined by
\eqref{eq:AffineBoost}.
%
For a given projective boost, 
denote by $x^+(\gamma), x^0(\gamma), x^-(\gamma),$  and $\sigma(\gamma)$ the points going to 
$x^+, x^0, x^-$  and the subspace $\sigma$ when $\gamma$ is put in standard form.
The span 
\[ 
\W^{ws}(\gamma) :=  \o{span}\big( x^-(\gamma) \cup \sigma(\gamma)\big) \]  
is the subspace of $\SI^3$ extending the weak-stable plane $W^{ws}(\gamma)\subset\Lspace$.
Under the action of $\gamma^{n}$ as $n \ra +\infty$, 
images of points on $W^{ws}(\gamma)\subset\Lspace$ limit towards  $x^0(\gamma)$  along the line 
\[ \sigma(\gamma) - \{x^0(\gamma), x^0(\gamma)_-\} \] 
where $\sigma(\gamma)$ is the subspace spanned by $\Axis(\gamma)$ in $\Lspace$.

\subsubsection{The action of projective-boost automorphisms} 
Recall that $\SI^3$ is the quotient $\R^4 -\{O\}/\sim$ where 
$\sim$ is given by $\vv \sim \vw$ iff $\vv = s \vw$ for $s > 0$. 
Recall that homogeneous coordinate system of $\SI^3$ is given by 
setting a point $p$ of $\SI^3$ to have coordinates 
$\llrrparen{ x: y:  z: t}$ where $(x, y, z, t)$ divided by its Euclidean norm represent $p$ 
as a unit vector in $\SI^3$. 
We identified $\Lspace$ with the open affine space given by $t >0$.
We use the coordinates so that $\SI^2_{\infty}$ corresponds to the hyperplane given by $t = 0$
and the origin $O$ of $\Lspace$ is given by $\be_4$.
The closure of $\Lspace$ is the $3$-dimensional hemisphere $\He$ with $\partial \He = \SI^2_{\infty}$. 
For this system of coordinates, $(x, y, z) \in \Lspace$ is given 
homogeneous coordinates $\llrrparen{ x: y:  z: 1}$ and coordinates in $\SI^2_{\infty}$
denoted $\llrrparen{ x: y:  z: 0}$.

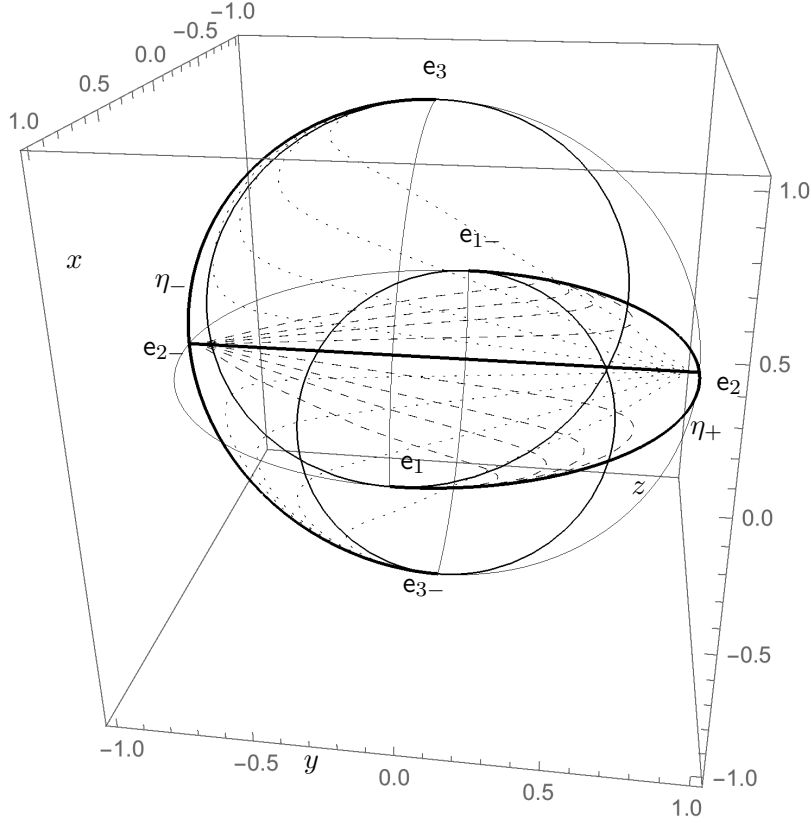
\begin{figure}[t!] 
\labellist
\small\hair 2pt
\pinlabel $\be_3$ at 190 335
\pinlabel $\be_{3-}$ at 185 105
\pinlabel $\be_2$ at 320 195
\pinlabel $\be_{2-}$ at 70 210
\pinlabel $\eta_+$ at 310 175
\pinlabel $\eta_-$ at 73 240
\pinlabel $\be_1$ at 180 160
\pinlabel $\be_{1-}$ at 210 260
\pinlabel $x$ at 30 250
\pinlabel $y$ at 135 27
\pinlabel $z$ at  280 150
\endlabellist
\centering

\centerline{\includegraphics[height=11cm]{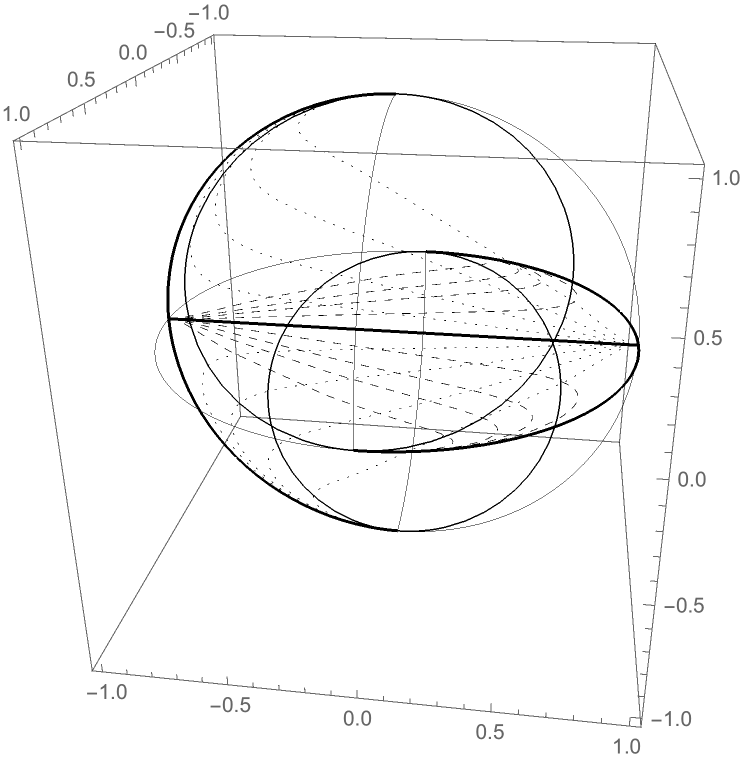}}
\caption{The action of a projective boost $\hat g$ 
on the $3$-hemisphere $\He$ where the boundary sphere 
$\SI^2_{\infty}$ is the unit sphere with center 
$(0,0,0)$ here as mapped by the normalization map \eqref{eqn:normalization}. 
The mathematica file is given at \cite{act2}. 
\label{fig:action}}

\end{figure}


Recall from \S \ref{subsec:pboost},
a projective automorphism $g$ of form 
\begin{equation}
\left[\begin{array}{cccc} 
\lambda & 0 & 0 &0  \\ 
0          & 1 & 0 & k \\ 
0          & 0 & \frac{1}{\lambda} & 0  \\ 
0          & 0 & 0 & 1
\end{array} \right] \lambda > 1, k \ne 0 \label{eqn:qA}
\end{equation}
under a homogeneous coordinate system of $\SI^3$
is a projective-boost. We call the above matrix the 
{\em standard boost matrix} with parameters $(\lambda,k)$.
In this coordinate system, $g$ acts on an open $3$-dimensional hemisphere, 
to be identified with $\Lspace$ given by $t > 0$,  
leaving invariant the following subspaces: 
\begin{align*}
\Axis(g) & = \sigma_{g} \cap \Lspace,  & \sigma_{g} :=\{\llrrparen{ x: y: z: t } | x =0, z=0\}, \\
W^{ws}(g)  & = S_g \cap \Lspace,  & \W^{ws}=S_{g} := \{\llrrparen{ x: y: z: t }| x =0 \}, \\
W^{wu}(g)  & = U_g \cap \Lspace, &  U_{g} := \{ \llrrparen{ x: y: z: t } | z=0\}.
\end{align*}
On $S_g \cap \He^o$, $g$ contracts vectors along a direction by a constant $1/\lambda$
and $g$ expands the vectors along a direction by a constant $\lambda$ on 
\[
\W^{ws}(g) \cap \He^o= W^{ws}(g) \]
assuming $\lambda > 1$. 
The subspace $S_g$ given by $x =0$ is the {\em weak-stable subspace}, 
and the subspace $U_g$ given by $z=0$ is the {\em weak-unstable subspace}.
(See \S \ref{sec:StandardBoostForm}.)


The great semicircles 
\begin{align*}
\eta_+ & := \ovl{\be_{1} \be_{2}(\be_{1-})} \\
\eta_- &:= \ovl{\be_{3}(\be_{{2}-})(\be_{{3}-})}
\end{align*}
are the the {\em attracting arc,\/} and {\em repelling arc\/}
for $g$, respectively. Their endpoints $\be_1, \be_{1-}, \be_3, \be_{3-}$
lie on the null latitudes of $\SI^2_{\infty}$, and both $\eta_+^{o}$ and $\eta_-^{o}$
are fibers of the collapsing map 
\[ \Ss_{0} \xrightarrow{\varepsilon} \rpo.\] 



\subsection{Sequences of projective boosts}

Let $\{\gamma_n\}$ be a sequence of projective boosts in $\Aut(\SI^3)$. 
A {\em rank} of an element of $\Ss(\Mat_{4}(\R))$ is the rank of any linear map $\R^4 \ra \R^4$ representing it. 
The compactness of $\Ss(\Mat_{4}(\R))$ implies that convergent subsequences exist.
Suppose that 
\[ \gamma_\infty = \lim_{n\to+\infty} \gamma^n \] 
is a projective endomorphism in $\Ss(\Mat_{4}(\R))$ of rank one.
(In general, the limit might have rank $2$ or $3$.)
Then the  undefined hyperplane of $\gamma_\infty$ equals 
a subspace $\W^{ws}_\infty$ of codimension one and the image of $\gamma_\infty$ equals the pair $x^+_{\infty}, x^+_{\infty-}$ where
\begin{align*}
\W^{ws}(\gamma_n) &\longrightarrow \W^{ws}_\infty, \\
x^+(\gamma_n) &\longrightarrow x^+_\infty
\end{align*}
as $n\longrightarrow +\infty$.

More generally, suppose that $\{\gamma_n\}$ is a sequence in $\Aut(\SI^3)$
converging to a projective endomorphism $\gamma_\infty$ in $\Ss(\Mat_{4}(\R))$ of rank one. 
Let $N(\gamma_\infty)$ denote the kernel of  $\gamma_\infty$. 
We can choose a subsequence of $\{\gamma_n\}$ so that the following hold: 
\begin{itemize}
\item the corresponding subsequence of the attracting fixed points $\{x^+(\gamma_n)\}$ 
converges to a point in the image $I(\gamma_\infty)$ and, 
\item the corresponding subsequence of the extended weak-stable subspaces  
$\{\W^{ws}(\gamma_n)\}$ converges to the undefined subspace $\Ss(N(\gamma_\infty))$. 
\end{itemize} 

\subsubsection{Convergence Lemma}
We will prove the following lemma for a properly convex $K$ because the proof is much simpler. 
But this is true for a general compact subset $K$ by the proof of Proposition \ref{prop:propdisc}.     
\begin{lem} \label{lem:Pi0}
Let $g_{\lambda, k}$ denote the projective boost on $\He$ defined by the standard boost matrix in \eqref{eqn:qA}
for homogeneous coordinates with coordinate functions $x, y, z, t$  
and let $\SI^2_0 = \W^{ws}(g_{\lambda, k})$ denote the subspace given by $x = 0$.
We assume that $k \geq 0, \lambda > 0$.
Then as $\lambda, k \ra +\infty$ where $k/\lambda \ra 0$\,{\rm :}
\begin{itemize} 
\item[(a)] $g_{\lambda, k}| \He - \SI^2_0$ converges  in the compact-open topology 
to a projective endomorphism 
\[ \llrrparen{ x: y: z: t}\xrightarrow{\Pi_0} 
\llrrparen{ {\mathsf{sgn}}(x): 0: 0: 0} = \be_{1} \hbox{ or } \be_{1-}, t \ge 0.\]  
\item[(b)] 
$g_{\lambda, k}|(\SI^2_0 \cap \He) - \eta_-$ converges 
to a projective endomorphism 
\[ \llrrparen{ 0: y: z: t} \xrightarrow{\Pi_1} \be_2. \]
\item[(c)] 
For a properly convex compact set $K$ in $\He - \eta_-$, 
a subsequence of $\{g_{\lambda, k}(K)\}$ converges to 
\begin{itemize} 
\item $\eta_+ =  \ovl{\be_{1}\be_{2}\be_{1-}}$ provided $K$ meets both components of 
\[\He -(\SI^2_0 \cap \He)\]    
\item otherwise to 
\begin{itemize} 
\item[(i)] $\{\be_2\}$ if $K \subset (\SI^2_0 \cap \He) - \eta_-$, 
\item[(ii)] $\{\be_{1}\}$ or $\{\be_{{1-}}\}$ if $K \cap (\SI^2_0 \cap \He) = \emp$, or else
\item[(iii)] one of the segments
\[ \ovl{\be_{1}\be_{2}}, \ovl{\be_{1-}\be_{{2}}}.\]
\end{itemize} 
\end{itemize} 
\end{itemize}
\end{lem}
\begin{proof} 
The first two items follow by normalizing the above matrix by dividing it by the maximal norm of the entries. 


 
So, we now suppose that $K \cap \SI^2_0 \ne \emp$, and 
$K$ meets both components of $\He - \SI^2_0$. 
For a properly convex compact  subset $K$ in $\He - \eta_-$, 
let $K_1$ be  the nonempty properly convex compact set $K \cap \SI^2_0 - \eta_-$.
We can take a closed $2$-dimensional hemisphere $H$ with boundary $\partial H \subset \SI^2_{\infty}$  
so that a component $H'$ of $\He - H$ satisfies 
\begin{align*}
K & \subset H', \\  
H' \cap \eta_{-}&= \emp, \\
H & = \partial_{\He} H'
\end{align*} 
which is the boundary of $H'$ in the relative topology of $\He$. 
%

We can operate as below:
\begin{itemize}
\item there exist a properly convex compact domain $D_1$ in a component of $H' - \SI^2_0$ with $x > 0$ and
another one $D_2$ in a component of $H' - \SI^2_0$ with $x < 0$ so that 
\begin{align*}
 K \subset J &:=\Join(D_1, D_2), \\
 J \cap \eta_- &= \emp. \end{align*}
\item For \[ K_2:= J \cap \SI^2_0 \supset K_1, \] 
the join $J_1:= \Join(K_2, D_1)$ and the join $J_2:= \Join(K_2, D_2)$
satisfy $J = J_1 \cup J_2$. 
\end{itemize}

By the second item of  
Lemma~\ref{lem:Pi0}, 
$\{g_{\lambda, k}(K_2)\}$ converges to $\be_{2}$ as $\lambda, k \ra \infty$.
Since 
\[\{g_{\lambda, k}(D_1) \} \ra \{\be_{1}\} \hbox{ and } \{g_{\lambda, k}(D_2)\} \ra \{\be_{{1-}}\}\] 
as $\lambda, k \ra \infty$ by above, 
\begin{itemize}
\item $\{g_{\lambda, k}(J_1)\} \ra s_1 = \ovl{\be_{1} \be_{2}}$ in $\SI^2_{\infty}$
and  
\item $\{g_{\lambda, k}(J_2)\} \ra s_2 = \ovl{\be_{{1-}}\be_{2}}$ in $\SI^2_{\infty}$ hold
by Lemma \ref{lem:geoconv} respectively. 
\end{itemize}
Hence, we conclude $\{g_{\lambda, k}(J)\} \ra \eta_+$.  

Let $p_{1}\in K_{1}$ and let $p_{+}\in D_{1}, p_{-} \in D_{2}$ be two points. 
Then 
\[\{g_{\lambda, k} (\ovl{p_{+}p_{1}} \cup \ovl{p_{-}p_{1}}) \} \ra \eta_{+} \hbox{ as } \lambda, k \ra \infty.\] 
Since \[ \ovl{p_{+}p_{1}} \cup \ovl{p_{-}p_{1}}\subset  K \subset J, \]
we obtain $\{g_{\lambda, k}(K)\} \ra \eta_{+}$ as $\lambda, k \ra \infty$
by Section \ref{subsec:space}.


So, we now suppose that $K \cap \SI^2_0 \ne \emp$ and 
$K$ meets only one component of $\He - \SI^2_0$. Then 
the arguments are simpler and $\{g_{\lambda, k}(K)\}$ converges to 
$\ovl{\be_{1}\be_{2}} \hbox{ or } \ovl{\be_{1-}\be_{{2}}}$
as $\lambda, k \ra \infty$. 

\end{proof} 


We remark that our proof of Proposition 
\ref{prop:propdisc} generalizes the above one,
using dynamical properties of Margulis spacetimes.



\section{Dynamics and hyperbolic geometry} \label{sec:hyps}

\subsection{$\rptwo$-structures on surfaces} \label{subsec:rpt}


Let $\vth$ correspond to an element $A$ of $\SL(3, \R)$ diagonalizable with three distinct positive 
eigenvalues. 
Then the induced $\vth' \in \Aut(\SI^2_{\infty})$ has six fixed points $a, a_-, r, r_-, s, s_-$
and three great circles $l_1, l_2,$ and $l_3$ so that 
\begin{align*}
\{a, r, a_-, r_-\} &\subset l_1, \\
\{r, s, r_-, s_-\} & \subset l_2, \\
\{a, s, a_-, s_-\} &\subset l_3 .\end{align*}
We assume that $a$ corresponds to the largest eigenvalue and $r$ to the smallest one. 
Eight invariant open triangles compose
$\SI^2_{\infty}  - ( l_1 \cup l_2 \cup l_3)$. 
Let $\tri$ be the triangle with vertices $a, r, s$. Then 
\begin{align*}
(\tri \cup \ovl{ar}^o \cup \ovl{a s}^o)&/\langle \vth' \rangle, \\
(\tri \cup \ovl{ar}^o \cup \ovl{r s}^o)&/\langle \vth' \rangle \end{align*}
are both examples of compact annuli. As quotients of domains of $\SI^2_{\infty}$, 
they have $\rptwo$-structures. (See \S \ref{sub:rpnstr}.) 
Such an $\rptwo$-surface-with-boundary is called 
an {\em elementary annulus.} 
%
Let $\tri'$ be an adjacent triangle sharing $\ovl{ar}^o$ with $\tri$ in the boundary, 
and $\tri''$ be one sharing $\ovl{as}^{o}$ with $\tri$ in the boundary.  
Then 
\begin{align}\label{eqn:piann}
 (\tri \cup \tri' \cup \ovl{ar}^o \cup \ovl{as}^o \cup \ovl{as_-}^o) &/\langle \vth' \rangle,  \notag \\
 (\tri \cup \tri'' \cup \ovl{ar}^o \cup \ovl{ar_-}^o \cup \ovl{as}^o)  & /\langle \vth' \rangle 
 \end{align}
are examples of compact annuli again. 
An $\rptwo$-surface projectively diffeomorphic 
to one of these (for a choice of $\vth$ as above) is said to be a {\em $\pi$-annulus}. 
They are the union of 
two elementary annuli meeting at a boundary component.  
 (See \cite{cdcr1} and \cite{cdcr2} for more details.)


Let $\tilde S$ be the universal cover of an $\rptwo$-surface $S$ with 
a developing map $\dev:\tilde S \ra \rptwo$ and a holonomy homomorphism $h$.
For a circle $\SI^1$, 
a {\em closed geodesic} $c: \SI^1 \ra S$ 
is a closed curve where 
$\dev\circ \tilde c$ is a straight arc in $\rptwo$ for a lift $\tilde c: \R \ra \tilde S$ of $c$.
A closed geodesic is {\em principal} if for a lift $\tilde c$ to the universal cover $\tilde S$,
$\dev\circ \tilde c$ is an embedding to 
a straight arc connecting an attracting and a repelling fixed point
of $h(\gamma)$ for a deck transformation $\gamma$ of $\tilde S$
satisfying $\tilde c(t +2\pi) = \gamma \circ \tilde c(t)$.
(Here $\gamma$ is said to be the {\em corresponding} deck transformation of 
$\tilde c$ and $c$, and it exists uniquely.)


A disjoint collection  $c_1, \dots, c_m$ of simple closed geodesics {\em decomposes}
an $\rptwo$-surface $S$ into subsurfaces $S_1, .., S_n$ if each $S_i$ is the closure of a component of 
$S - \bigcup_{i=1,..,m} c_i$ where we do not allow a curve $c_i$ to have two one-sided neighborhoods in only one 
$S_j$ for some $j$.

In \cite{cdcr2}, we proved:
\begin{thm} \label{thm:dec}
Let $\Sigma$ be a compact orientable $\rptwo$-surface with principal 
geodesic or empty boundary and $\chi(\Sigma) < 0$. 
Then $\Sigma$ has a collection of disjoint simple closed principal geodesics
decomposing $\Sigma$ into properly convex 
$\rptwo$-surfaces with principal geodesic boundary and of negative Euler characteristic 
and/or $\pi$-annuli with principal geodesic boundary. 
\end{thm}


\subsection{Complete hyperbolic surfaces and Fuchsian groups} \label{subsec:Fuch} 

Let $G$ be a discrete subgroup of projective automorphisms of $\Ss_+$ in $\Aut(\SI^2_{\infty})$. 
We assume that $G$ has no elliptic or parabolic elements. 
Here complete geodesics of the hyperbolic metrics are maximal straight lines in $\Ss_+$
and vice versa.

We suppose that $G$ is nonelementary; that is, it does not act on a pair of points 
in $\clo(\Ss_+)$. Then $\Ss_+/G$ has a complete hyperbolic metric induced from 
the Hilbert one. $\Ss_+/G$ is {\em geometrically finite} in the sense that 
$\Ss_+/G$ contains a compact surface bounded by closed geodesics that is the deformation retract of $\Ss_+/G$. 

Let $\Lambda$ denote the limit set of an orbit $G(x)$ for a point $x \in \Ss_+$, 
Since $G$ is a geometrically finite Fuchsian group,
$\Lambda$ is a Cantor subset of $\partial\Ss_+$ of Lebesgue measure zero 
and contains  the dense set of fixed points of nonidentity elements of $G$. 
Let us define $\Sigma_+ := \Ss_+/G$. 
We define $\tSigma'_+$ as $\clo(\Ss_+) - \Lambda$. 
(See Chapter 8 of Beardon \cite{Beardon} for the classical Fuchsian group theory used here.) 

Let  ${\mathcal{J}}$ be the set indexing the boundary components of $\tilde \Sigma'$. 
We denote the boundary components of $\tilde \Sigma'_+$ by $\partial_i \Ss_{+}, i \in {\mathcal{J}}$. Hence, 
\[ \tilde \Sigma'_+ = \Ss_+ \cup \bigcup_{i\in \mathcal{J}} \partial_i\Ss_{+} = \clo(\Ss_+) - \Lambda. \]
We obtain a compact surface  $\Sigma'_+  = \tSigma'_+/G$ containing open surface $\Sigma_+$ 
compactifying $\Sigma_+$ with boundary components diffeomorphic to circles. 

Note that for each element $g \in G$, either $g(\partial_i \Ss_{+}) \cap \partial_i \Ss_{+} = \emp$ 
or $g(\partial_i\Ss_{+}) = \partial_i \Ss_{+} $ holds with $g$ in the cyclic group generated by the deck transformation 
corresponding to $\partial_i\Ss_{+}$. 

We obtain a convex domain 
$\Omega_+$ closed in $\Ss_+$ bounded by a union of straight segments ${\mathbf{l}}_i$ for $i \in {\mathcal{J}}$.
Each ${\mathbf{l}}_i$ is a geodesic in $\Ss_+$ connecting  
ideal endpoints of $\partial_i\Ss_{+}$.
We choose a primitive element $\bg_i \in \Gamma$ acting on ${\mathbf{l}}_i$ for each $i \in {\mathcal{J}}$. 
Since $\Gamma$ acts as a geometrically finite Fuchsian group on $\Ss_+$ without parabolics, 
$\Omega_+/\Gamma$ is a compact hyperbolic surface with geodesic boundary 
that is the convex hull of $\Ss_+/\Gamma$ and homeomorphic to $\Sigma'_+$. 
Thus, \[\Ss_+ - \Omega_+ = \coprod_{i \in  {\mathcal{J}}} D_i \] 
where $D_i$ is a convex open domain $\partial D_i = \clo(\partial_i\Ss_{+}) \cup {\mathbf{l}}_i$ for $i \in {\mathcal{J}}$
that covers an open annulus in $\Sigma'_+$. 

Since ${\mathbf{l}}_i $ covers a simple closed geodesic in $\Sigma_+$,  
\begin{align} \label{eqn:disj}
g(\mathbf{l}_i) \cap \mathbf{l}_i = \emptyset & \text{~if~}  i \neq j,   \  g \in\Gamma,  \notag \\
g(\mathbf{l})_i \cap \mathbf{l}_i = \emptyset & \text{~if~}  g\notin \mathsf{Stab}({l}_i). \end{align}

\subsection{The dynamics of geodesic flows on hyperbolic surfaces} \label{subsec:dyn} 

Now we assume that $\Gamma$ also acts freely and properly on $\Lspace$.
Assume that ${\L}(\Gamma) \subset \SO(2, \R)^o$. 
We recall that $\Gamma$ is isomorphic to a free group of finite rank $\geq 2$, and $\Ss_+/\Gamma$ is  
a complete genus $\tilde{\bg}$ hyperbolic surface with $\mathsf b$ ideal boundary components
and without parabolics.   
Recall that $\Sigma_+$ is the interior of the surface $\Sigma'_+$.

The Fuchsian $\Gamma$-action on the boundary $\partial\Ss_+$ of the standard disk $\Ss_+$ in $\SI^2_{\infty}$ 
forms a {\em discrete convergence group}\/:
For every
sequence $g_j$ of mutually distinct elements of $\Gamma$, 
there is a subsequence $g_{j_k}$ and (not necessarily distinct)
points $ a, b$ in the circle $\partial\Ss_+$ such that 
\begin{itemize} 
\item the sequences $g_{j_k} (x) \ra a$ locally uniformly in $\partial\Ss_+- \{b\}$, that is, 
uniformly in any compact subset of $\partial\Ss_+ - \{b\}$, 
and 
\item $g^{-1}_{j_k} (y) \ra b$ locally uniformly on $\partial\Ss_+ - \{a\}$ respectively as $k \ra \infty$. 
\end{itemize}
(See \cite{ABT} for details.)
We remark that 
\begin{equation} \label{eqn:ab} 
a, b \in \Lambda.
\end{equation}
For later purpose, we say that $a$ is an {\em attractor point} and $b$ is a {\em repeller point} of 
the sequence $\{g_{j_k}\}$.
Let $a_i$ and $r_i$ denote the attracting fixed point and the repelling fixed point of $g_i$ respectively.

\begin{lem}\label{lem:fix}
Let $g_{j_k}$ be a sequence satisfying the above convergence group properties. 
Suppose that the attractor point $a$ is distinct from the repeller point $b$. 
Then it follows that 
\begin{equation} \label{eqn:convg}
\{a_{j_k}\} \ra a \hbox{ and } \{r_{j_k}\} \ra b. 
\end{equation}
\end{lem} 
\begin{proof} 
One can use the fact that if a continuous map $f$ sends a closed arc $I$, $I \subset \partial \Ss_{+}$, 
to its interior $I^o$, then the fixed points
are in the image $f(I)$.
(See also the $\eps$-hyperbolicity in p.256 of \cite{AMS}. Since $a \ne b$, these transformations are uniformly 
$\eps$-hyperbolic.)
\end{proof}

Let $\Uu\Ss_+$ be the unit tangent vector bundle over $\Ss_+$,
and let $\Uu\Ss_+/\Gamma = \Uu \Sigma_+$ be the unit tangent vector bundle over $\Ss_+/\Gamma = \Sigma_+$.
By following the geodesics in $\Sigma_+$, we obtain a 
geodesic flow 
\[\Phi: \Uu \Sigma_+ \times \R \ra \Uu \Sigma_+.\]
A {\em geodesic current} is a Borel probability measure on $\Uu \Ss_+/\Gamma$ invariant under the geodesic flow
which is supported on a union of nonwandering geodesics. 
Let ${\mathcal C}(\Ss_{+}/\Gamma)$ denote the space of  the space of all geodesic
currents with the weak-$\ast$ topology.


 Let $\Uu_{\rm rec} \Sigma_+$ denote the set of unit tangent vectors of 
nonwandering  geodesics on $\Sigma_+$. 
We recall 
\begin{lem} \label{lem:glm}
Let $\Sigma_+$ be as above. 
Then 
\begin{itemize}
\item $\Uu_{\rm rec} \Sigma_+ \subset \Uu\Sigma_+$ is a connected, compact, geodesic flow 
invariant set. 
\item The inverse image $\Uu_{\rm rec} \Ss_+$ of $\Uu_{\rm rec} \Sigma_+$ in $\Uu_{\rm rec} \Ss_+$
is precisely the set of geodesics with both endpoints in $\Lambda$.
\item The set 
\begin{align*} 
\left\{(a, r) \in \partial \Ss_+ \times \partial \Ss_- \right| \, & a \hbox{ is an attracting fixed point and } \nonumber \\
& r \hbox{ is a repelling fixed point of } g \in \Gamma \big\}
\end{align*} 
 is dense in $\Lambda \times \Lambda$.
\end{itemize} 
\end{lem} 
\begin{proof}
The first item is in Lemma 1.2 in \cite{GLM} and 
the second item is in the proof of the same lemma. 
The connectedness is also proved in Lemma 1.3 in \cite{GLM}. 
The third item follows from the fact that the set of closed geodesic orbits are dense in $\Uu_{\rm rec} \Sigma_+$
as the flow is Anosov (see \cite{Katok2}).
\end{proof} 

Let $\Delta(\Lambda)$ denote the diagonal of $\Lambda \times \Lambda$. 
Such pairs $(\lambda_u,\lambda_s)\in\Lambda\times\Lambda - \Delta(\Lambda)$ correspond
to nonwandering geodesics in the quotient hyperbolic surface.
Thus, the set of closed geodesics is dense in the set of nonwandering geodesics 
by Lemma \ref{lem:glm}. 

We have a compactification picture of the above phenomena: 
For a nonelementary Fuchsian group $\Gamma_0\subset\PGL(2, \R)$, 
the closure in the projective space $\P\big(\Mat_2(\R)\big)$ equals
$\Gamma_0 \cup (\Lambda\times\Lambda)$, where $\Lambda\subset\partial\Ss_+$
is the limit set of $\Gamma_0$. Here we could write
\[
\P\big(\Mat_2(\R)\big) \ =\  \PGL(2,\R) \cup (\rpone\times\rpone)
\]
where $\rpone\times\rpone$ corresponds to the set of projective equivalence classes
of rank-one $2\times 2$ real matrices. Projective equivalence classes
of such matrices are completely determined by their kernels and images,
which are {\em arbitrary\/} lines through the origin in $\R^2$.  

\section{Real projective Schottky uniformizations} 
\label{sec:bord}


Goldman constructed an $\rptwo$-surface $\Sigma$ with a free holonomy group
in \cite{Gf}. The holonomy group is in the image $\mathsf{PSO}(2, 1)$ of $\SOto$ in $\PGL(3, \R)$. 
We will show that a domain in $\SI^2_{\infty}$ regularly covers $\Sigma$. 

Let $\Gamma$ be a free group of rank $\geq 2$ in $\Aut(\SI^3)$ acting  
on $\Ss_+ \cup \Ss_-$ properly and freely
but without parabolics.
We do not require $\Gamma$ to act properly and freely on $\Lspace$ in this section. 
(We will use the notation of \S \ref{subsec:Fuch} letting $G = \Gamma$. )

\subsection{The construction of the $\Gamma$-invariant domain}
Recalling \S \ref{subsec:Fuch}, 
we assume $\L(\Gamma) \subset \SOto^o$ initially
and so $\Gamma$ acts on $\Ss_+$. 
Then $\Ss_+/\Gamma$ is 
a complete genus $\tilde{\bg}$ hyperbolic surface with $\mathsf b$ ideal boundary components
where ${\mathsf b} \geq 1$ and $\tilde{\bg}$ is just some integer $\geq 0$.
In other words, $\Ss_+/\Gamma$ can be compactified by adding 
boundary components to a compact surface $\Sigma'_+$ with $\mathsf b$ boundary components. 
The universal cover $\tSigma'_+$ of $\Sigma'_+$ is identified with the union of a domain $\Ss_+$ and
a collection of open arcs in $\partial\Ss_+$. 
(At the moment, we do not require that $\Gamma$ acts properly or 
freely on $\Lspace$. )

Let $\Lambda$ denote the limit set of an orbit $\Gamma(x)$ for $x \in \Ss_+$. 
We identify $\tSigma'_+$ as $\clo(\Ss_+) - \Lambda$. 
Denote the components of $\partial\Sigma'$ by 
$\partial_i\Sigma'$ where $i=1,\dots,\bbb$.
\[ \tilde \Sigma'_+ = \Ss_+ \cup \bigcup_{i\in \mathcal{J}}\partial_i\tilde \Sigma'_+.\]

We will be using the same symbols for the group $\Gamma$ in $\Isom^+(\Lspace)$ its extension to $\SI^3$ and 
its element $g$ and its extension to $\SI^3$.
We identify $\SI^2_{\infty}$ with the boundary of $\He$, called the {\em boundary sphere}. 
$\Gamma$ acts on the standard circles $\partial \Ss_\pm$ and the interiors $\Ss_\pm$ in $\SI^2_{\infty}$.
Then each nonidentity element $g$ of $\Gamma$ has an attracting fixed point $a$ in $\partial\Ss_+$, 
a repelling fixed point $r$ in $\partial\Ss_+$, and a saddle-type fixed point $s$ in $\Ss_0$. Of course, their antipodes in $\SI^2_{\infty}$ are also 
an attracting fixed point $a_-$ in $\partial \Ss_-$ and a repelling fixed point $r_-$ in $\partial \Ss_-$ 
and a saddle-type fixed point $s_-$ in $\Ss_0$ respectively. 
 
Since $g$ acts on $\partial\Ss_+$ and fixes $a$ and $r$, we obtain that 
$g$ acts on the two $1$-dimensional subspaces tangent to $\partial\Ss_+$ at $a$ and $r$. 
Hence, $\{s, s_-\}$ is their intersection, and also
\begin{equation}\label{eqn:saddle} 
\{s, s_-\} \subset \varepsilon(a) \cup \varepsilon(r) 
\end{equation} 
as $\varepsilon(a)$ and $\varepsilon(r)$ 
are disjoint halves of these $1$-dimensional subspaces.

We now define
$\tSigma'_- := {\A}(\tSigma'_{+})$.
Then $\Gamma$ acts properly and freely on $\tSigma'_-$ by antipodality,
and $\tSigma'_-/\Gamma$ is a compact $\rptwo$-surface with geodesic boundary diffeomorphic to $\Sigma'_+$. Denote it by $\Sigma'_-$. 
Also, let $\Omega_-$ be the set of antipodal points to $\Omega_+$. 
Then $\Gamma$ acts properly and freely on $\Omega_+$ and $\Omega_-$. 
Denote the respective compact quotient surfaces with geodesic boundary 
by $\Sigma''_+$ and $\Sigma''_-$. 

For each point $x$ of $\partial\Ss_+$, we recall that $\varepsilon(x)$ is an open segment of length $\pi$ with end points 
$x$ and $-x$ tangent to $\partial\Ss_+$. 
The open arcs $\varepsilon(x)$ for $x\in \partial\Ss_+$ give leaves of the foliation 
$\mathcal F$ in  $\Ss_0$. 
Let us give an induced orientation on $\Sigma'_+$ and hence an induced orientation on each ${\mathbf{l}}_i$ for $i \in {\mathcal{J}}$.
Let $p_i$ and $q_i$ denote the forward and backward end points of ${\mathbf{l}}_i$. 
We draw a segment $s_i = \varepsilon(p_i)$.
We also draw a segment $t_i = \varepsilon(q_i)$. 
Then ${\mathbf{l}}_i, s_i, t_i,$ and ${\mathbf{l}}_{i-}$ bound an open disk $E_i$ invariant under $\langle \bg_i \rangle$, which 
we call a {\em strip}.  
We denote by ${\mathcal{R}}_i$ the open strip union with ${\mathbf{l}}_i$ and ${\mathbf{l}}_{i-}$; that is, 
\[ {\mathcal{R}}_i = E_i \cup  {\mathbf{l}}_i \cup {\mathbf{l}}_{i-}, \partial E_i = {\mathbf{l}}_i \cup s_i \cup t_i \cup {\mathbf{l}}_{i-}.\]
(See Figures \ref{fig:3dfig}, \ref{fig:ann}, and \ref{fig:figure2} for the pictures of these arcs tangent to $\Ss_+$ and $\Ss_-$.)

Using $\mathcal F$, we obtain:
\begin{prop} \label{prop:disj} 
The strips ${\mathcal{R}}_i$ and ${\mathcal{R}}_j$ are disjoint for $i \ne j$ where 
$i,j\in \mathcal{J}$. 
\end{prop}
\begin{proof}
Since $s_i, t_i, s_j, t_j$ are all 
leaves of $\mathcal F$ corresponding to distinct points of $\partial\Ss_+$,  the mutual disjointness
of these arcs follows. (See Figure \ref{fig:ann}.)
$s_i$ does not intersect $t_i, s_j, t_j$ for $i \ne j$, where $i, j \in \mathcal{J}$,
and $t_i$ does not intersect $s_i, s_j, t_j$ for $i \ne j$. 
Furthermore ${\mathbf{l}}_i$ and ${{\mathbf{l}}}_{i-}$ for 
$i \in \mathcal{J}$ are mutually disjoint
by \eqref{eqn:disj}.
Thus the boundary $\partial {\mathcal{R}}_i$ and $\partial {\mathcal{R}}_j$ for $i \ne j$ are disjoint. Hence, the conclusion follows.
\end{proof}




Goldman 
\cite{Gthesis} classified compact annuli with geodesic boundary and 
holonomy matrices that are diagonalizable with distinct positive eigenvalues
 following Nagano-Yagi \cite{NY}.
(See \cite{cdcr2} also.)

\begin{figure}[h]
\centerline{\includegraphics[height=6.5cm]{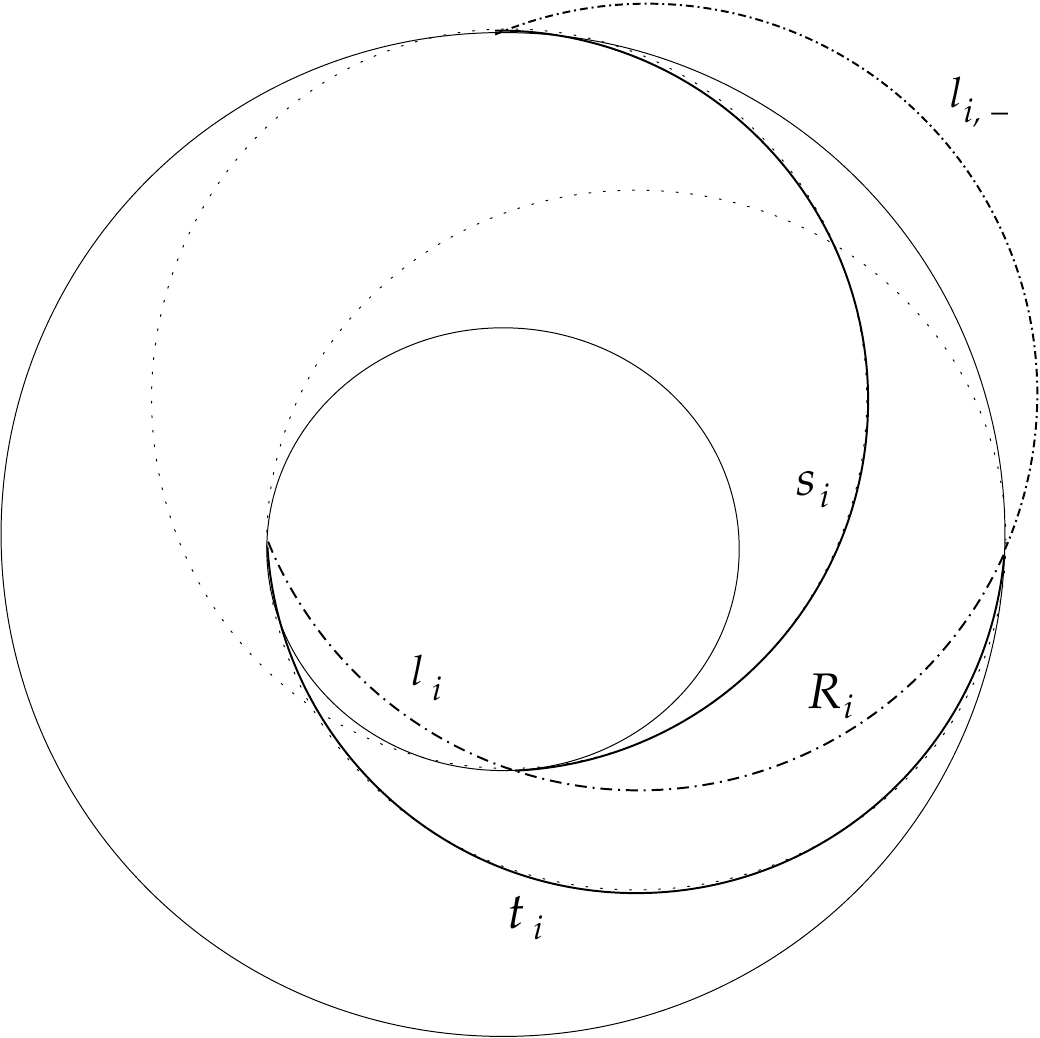}} 
\caption{The two stereographically projected  ``lune'' regions which cover $\pi$-annuli. \label{fig:ann}}
\end{figure}

\begin{prop} \label{prop:Ri}
Given ${\mathcal{R}}_i$ for 
$i\in {\mathcal J}$   
and $\gamma \in \Gamma$, we either have 
\[\gamma({\mathcal{R}}_i) = {\mathcal{R}}_i \hbox{ or } \gamma({\mathcal{R}}_i) \cap {\mathcal{R}}_i = \emp.\]
In the former case, $\gamma = \bg_i^n$, $n \in \bZ$ for 
the deck transformation $\bg_i$ corresponding to ${\mathbf{l}}_i$.
Furthermore,  
${\mathcal{R}}_i/\langle \bg_i \rangle$ is a compact annulus with geodesic boundary that 
decomposes into two $\pi$-annuli or four elementary annuli along some collection of simple closed geodesics. 
\end{prop}
\begin{proof}
Suppose that ${\mathcal{R}}_i \cap \gamma({\mathcal{R}}_i)\ne \emp$ holds. 
Then
${\mathbf{l}}_i \cap \gamma({\mathbf{l}}_i) \ne \emp$ holds since the arcs $t_i$ are mutually disjoint. 
Hence, \eqref{eqn:disj} implies the result.


There is a subspace $l$ of dimension $1$, a geodesic circle, containing ${\mathbf{l}}_i$ and ${{\mathbf{l}}}_{i-}$ in
$\SI^2_{\infty}$. $l$ and $t_i$ bound a closed disk $D_1$ and if we remove $t_i$ and the fixed points 
at the ends of ${\mathbf{l}}_i$, this region covers a $\pi$-annulus with principal boundary components.
The similar statement is true for the disk bounded by $l$ and $s_i$. 
Therefore, $A_i= {\mathcal{R}}_i /\langle \bg_i \rangle$ is a union of two compact $\pi$-annuli. In other words, 
$l - \clo({\mathbf{l}}_i) - \clo({\mathbf{l}}_{i-})$ maps to a simple closed principal geodesic decomposing $A_i$ into 
two $\pi$-annuli with principal boundary. (See Figure \ref{fig:ann}.)
Also, each $\pi$-annulus decomposes into two elementary annuli. (See \S \ref{subsec:rpt}.) 
\end{proof}


We say that for 
$i, j \in {\mathcal{J}}$, 
the annulus ${\mathcal{R}}_i/\langle \bg_i \rangle$ is {\em equivalent} to ${\mathcal{R}}_j/\langle \bg_j \rangle$
if ${\mathcal{R}}_j = g({\mathcal{R}}_i)$ and $g \bg_i g^{-1} = \bg_j^{\pm}$ for $g \in \Gamma$.
Thus, in fact, there are only $\bb$ equivalence classes of annuli of above form.

We define \[{\mathcal{A}}_i := {\mathcal{R}}_i \cap \Ss_0 = \bigcup_{x \in {\partial_i \Ss_{+}}} \varepsilon(x) 
\hbox{ for } i \in {\mathcal J}.\]
We note that ${\mathcal{A}}_i \subset {\mathcal{R}}_i$ for each 
$i \in {\mathcal{J}}$.
We finally define an open domain in $\SI^2_{\infty}$:
\begin{eqnarray} 
\tSigma  &=& \tSigma'_+ \cup \coprod_{i \in {\mathcal{J}}} {\mathcal{R}}_i \cup \tSigma'_- \nonumber \\ 
               &=& \tSigma'_+ \cup \coprod_{i \in {\mathcal{J}}} {\mathcal{A}}_i \cup \tSigma'_- \nonumber \\
               &=& \Omega_+ \cup  \coprod_{i \in {\mathcal{J}}} {\mathcal{R}}_i \cup \Omega_- \label{eqn:tsigma1} \\
               &=& \SI^2_{\infty} -  \bigcup_{x\in \Lambda} \clo(\varepsilon(x)). \label{eqn:tsigma}
\end{eqnarray}               
Since the collection whose elements are of form ${\mathcal{R}}_i$ mapped to itself by $\Gamma$, we showed that 
$\Gamma$ acts on this open domain.

\subsection{The $\rptwo$-surface}
First, $\tSigma$ does not contain any fixed point: 
Suppose that $g \in \Gamma$ acts on ${\mathbf{l}}_i$ for some $i \in {\mathcal{J}}$. 
Then $g$ acts on ${\mathcal{R}}_i$ and 
the four attracting and repelling fixed points are the vertices of ${\mathcal{R}}_i$ 
and two saddle-type fixed points are in $s_i$ or $t_i$. Hence, they are outside $\tSigma$.

Suppose that $g \in \Gamma -\{\Idd\}$ does not act on any of the boundary components. 
Then $g$ has four attracting and repelling fixed points $a,  r \in \partial\Ss_+$ and  $a_-, r_- \in \partial\Ss_-$. 
By \eqref{eqn:saddle}, 
saddle-type fixed points $s$ and $s_-$ are either on 
a great segment $\varepsilon(a)$ or the other one $\varepsilon(r)$
tangent to $\partial\Ss_+$  at $a$ and $r$ respectively.
Since $a, r \in \Lambda$ holds, the fixed points are outside $\tSigma$
by \eqref{eqn:tsigma}.

%
\begin{figure}[hb]
\centerline{\includegraphics[height=10cm]{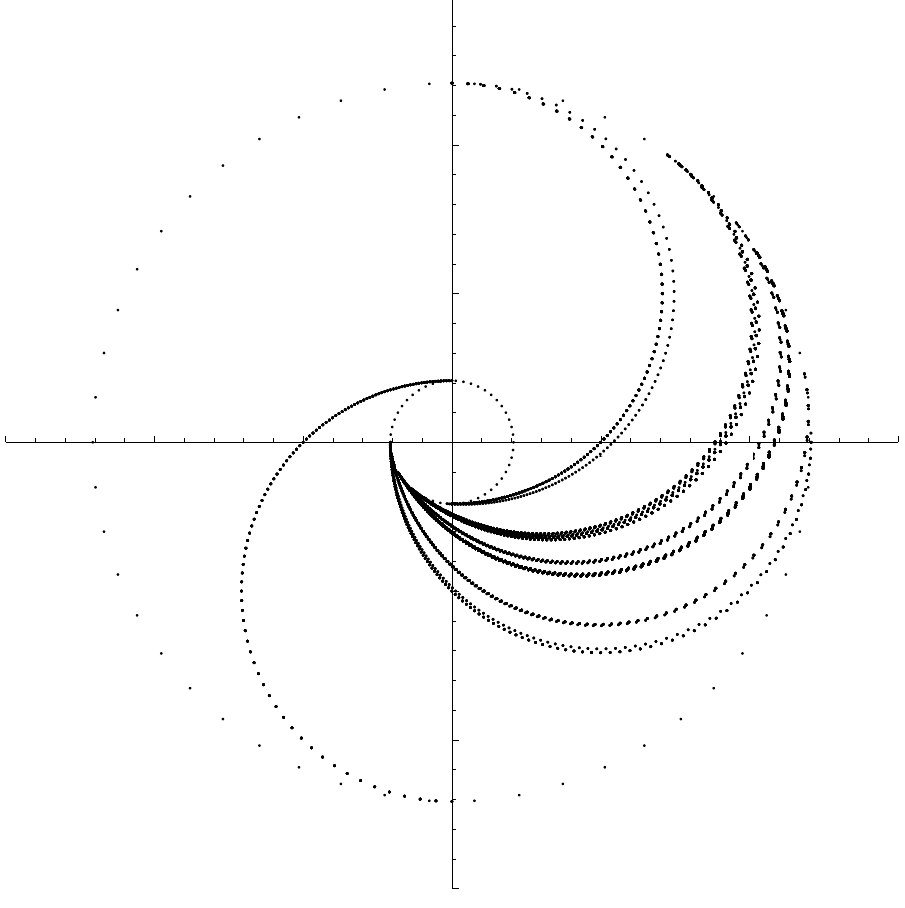}} 
\caption{ A figure representing an actual free group action with some of the arcs $s_i$ and $t_i$ drawn in point-plots on the stereographically projected sphere as in Figure~\ref{fig:figure2}.
}
\label{fig:figexamp}
\end{figure}

Let $G \subset \SOto^o$ be a subgroup whose elements preserve orientation 
on $\Ss_+$, and  any open domain $\mathcal{D}$ an open domain 
on $\SI^2_{\infty}$ upon which where $G$ acts freely and properly.
Then $\mathcal{D}/G$ has an induced orientation from $\Ss_+$.


\begin{thm} \label{thm:domainq} 
Let $\Gamma$ be an orientation-preserving finitely generated group in $\SOto^o$ without parabolics
acting freely and properly on $\Ss_+$ 
isomorphic to a free group of finite rank $\geq 2$. 
Denote the ideal boundary components of $\Ss_+/\Gamma$ by
$\partial_i(\Ss_+/\Gamma)$, for $i=1,\dots,\bbb$.
Then 
\begin{itemize} 
\item $\Gamma$ acts properly 
and freely on an open domain ${\mathcal{D}}$ in $\SI^2_{\infty}$ 
and ${\mathcal{D}}/\Gamma$ is homeomorphic to a closed surface of genus 
${\mathsf g}$, ${\mathsf g} \geq 2$. 
\item As an $\rptwo$-surface ${\mathcal{D}}/\Gamma$ decomposes along simple closed principal geodesics into a union of  $\bbb$ annuli $A_1, \ldots, A_{\mathsf b}$ and two convex $\rptwo$-surfaces $\Sigma''_+$ and $\Sigma''_-$.
Each annulus $A_i$, $i=1,\ldots, {\mathsf b}$ decomposes into two $\pi$-annuli.
\item $\Gamma$ also acts properly
and freely on ${\mathcal{D}}_- := {\mathscr A}({\mathcal D})$, 
and ${\mathcal{D}}_-/\Gamma$ is a closed $\rptwo$-surface diffeomorphic 
to ${\mathcal{D}}/\Gamma$.
\item The antipodal map ${\A}: \SI^2_{\infty} \ra \SI^2_{\infty}$ induces a projective diffeomorphism between  
${\mathcal{D}}/\Gamma$ and ${\mathcal{D}}_-/\Gamma$.
\end{itemize} 
\end{thm}
\begin{proof} 
We begin by letting ${\mathcal{D}} = \tSigma$ obtained as above in \eqref{eqn:tsigma1}. We showed that $\Gamma$ has no fixed point on ${\mathcal{D}}$. 
Let $\clo_\cD({{\mathcal{A}}_i})$ denote the closure of ${\mathcal{A}}_i$ in ${\mathcal{D}}$
for each $i\in \mathcal{J}$.

Now $\Gamma$ acts properly on $\tSigma'_+$ and $\tSigma'_-$ 
and each $g \in \Gamma$ sends $\clo_\cD({\mathcal{A}}_i )$ to $\clo_\cD({\mathcal{A}}_j)$ for $j \ne i$ or
$g$ is in the infinite cyclic subgroup of $\Gamma$ acting properly 
on $\clo_\cD({\mathcal{A}}_i) $ if $g(\clo_\cD({\mathcal{A}}_i) ) = \clo_\cD({\mathcal{A}}_i) $.
Hence, $\Gamma$ acts properly on 
$ \bigcup_{i\in \mathcal{J}}  \clo_\cD({\mathcal{A}}_i)$.

For any compact set $K\subset\tSigma$, 
the compact sets 
\begin{align*}
K_1 &:= K \cap \tSigma'_+, \\
K_{-1} &:= K \cap \tSigma'_-, \\
K_0 &:= K \cap \bigcup_{i\in \mathcal{J}}  \clo_\cD({\mathcal{A}}_i)
\end{align*}
satisfy: 
\[
K = K_+ \cup K_- \cup K_0. \] 


Since each of $\clo_\cD({\mathcal{A}}_i)$ has an open neighborhood $O_j$ forming a mutually disjoint collection of 
open sets covering $K_0$, it follows that 
$K_0$ meets only finitely many $\clo_\cD({\mathcal{A}}_i)$.

If $g(K) \cap K \ne \emp$, we have $g(K_i) \cap K_i \ne \emp$ for some $i = -1, 0, 1$.
Since $\Gamma$ acts properly on $\tilde \Sigma'_+$ and 
$\tilde \Sigma'_-$   and 
$ \bigcup_{i\in \mathcal{J}}  \clo_\cD({\mathcal{A}}_i)$ respectively, 
only finitely many elements $g$ of $\Gamma$ satisfy $g(K_i) \cap K_i \ne \emp$ for each $i$, $i=-1,0,1$.
We showed that $\Gamma$ acts properly and freely on $\mathcal D$. 

Since $\tSigma'_+/\Gamma$ and $\tSigma'_-/\Gamma$ are compact surfaces, 
and $ \bigcup_{i\in \mathcal{J}} \clo_\cD({\mathcal{A}}_i)/\Gamma$ is a union of 
finitely many compact annuli, ${\mathcal{D}}/\Gamma$ is a closed surface. 

For the second item, ${\mathcal{D}}/\Gamma$ decomposes along the images of 
the union of ${\mathbf{l}}_i, {\mathbf{l}}_{i-}, i\in {\mathcal{J}}$ 
to $\Sigma''_+$ and $\Sigma''_-$ and annuli with boundary components that are images of 
the union of ${\mathbf{l}}_i$ and ${\mathbf{l}}_{i-}$, $i \in {\mathcal{J}}$.
Proposition \ref{prop:Ri} implies that these annuli are obtained from two $\pi$-annuli.

Since 
$\A(\Gamma| \SI^{2}_{\infty}) \A^{-1} = \Gamma|\SI^{2}_{\infty}$ and 
${\A}({\mathcal{D}}) = {\mathcal{D}}_-$ hold, 
$\Gamma$ acts properly and freely on ${\mathcal{D}}_-$. 
Let $\kappa$ denote the induced diffeomorphism $\Sigma \ra {\mathcal{D}}_-/\Gamma$
where $\kappa$ sends $\Sigma''_-$ to $\Sigma''_+$ and vice versa.
$\kappa$ sends annuli in $\Sigma$ to ones in ${\mathcal{D}}_-/\Gamma$ but they do not share 
subdomains in the universal cover. (Actually ${\mathcal{R}}_i$ goes to its antipodal image 
which is distinct from ${\mathcal{R}}_i$ for each $i \in {\mathcal{J}}$.)

The rest follows from this observation. 
\end{proof}



\begin{proof}[The proof of Theorem \ref{thm:A}] 
If $\L(\Gamma) \subset \SO(2, 1)^o$, then Theorem \ref{thm:domainq} proves Theorem \ref{thm:A}. 
Now, we consider the case when $\L(\Gamma)$ is a subgroup of $\SOto$ in general. 
Define $\Gamma'$ to be the subgroup of $\Gamma$ of index $ 2$ so that 
${\L}(\Gamma') \subset \SO(2, 1)^o$ using above notations.

%
%

Let $\psi$ be the element of $\Gamma - \Gamma'$.
Then $\psi$ is orientation-preserving but $\psi(\Ss_+) = \Ss_-$. 
Let $\mathcal{D}'$ be defined by \eqref{eqn:tsigma1} for $\Gamma'$. 
Since $\psi$ preserves orientation,
$\psi$ acts on ${\mathcal D}'$.
Thus, $\psi$ acts on ${\mathcal D}'$ and hence $\Gamma$ acts on ${\mathcal D}'$. 

Recall from the introduction that ${\L'}(\Gamma)$ is a free group of rank $\geq 2$ acting 
freely on $\Ss_+$ as a subgroup of $\PGL(3, \R)$ and 
$\Ss_+/{\L'}(\Gamma)$ is a nonorientable complete hyperbolic surface. 
Since ${\L'}(\Gamma)$ acts freely on $\Ss_+$, 
$\psi$ goes to an element of $\PGL(3, \R)$ that acts freely on $\Ss_+$
and on $\bigcup_{i\in \mathcal{J}} \partial_i\Ss_+$. 
Hence, $\psi$ as an element of $\Aut(\SI^3)$ 
acts without fixed point on 
\[\bigcup_{i\in \mathcal{J}} \bigcup_{x \in \partial_i\Ss_+} \varepsilon(x) = 
\coprod_{i\in \mathcal{J}} {\mathcal{A}_i}\]
since $\psi$ does not act on any fiber of the fibration. 
Hence $\psi$ has no fixed point on 
\[\tSigma'_+ \cup \coprod_{i\in \mathcal{J}} {\mathcal{A}}_i \cup \tSigma'_- .\]
${\mathcal D}'/\Gamma'$ double-covers a closed surface ${\mathcal D}/\Gamma$.

Finally, it is straightforward to verify the same statements for ${\A}({\mathcal D})$. 
Replacing with this set is equivalent to replacing $\clo(\varepsilon(x))$ with $\clo(\varepsilon(x)_-)$ for every appropriate point $x$
in \eqref{eqn:tsigma}.


We prove the uniqueness under the assumption that $\Sigma$ has a regular covering domain $\mathcal{D}' \subset \SI^2_{\infty}$. 
Theorem R of \cite{Gf} implies that for a given Fuchsian holonomy group $\Gamma$, a closed $\rptwo$-surface $\Sigma$
is obtained as follows: Let $\Sigma'$ denote a closed properly convex $\rptwo$-surface covered by $\Ss_+$. 
We cut  along the union of mutually disjoint simple closed geodesics in $\Sigma'$, 
and insert $\rptwo$-annuli with principal geodesic boundary with conjugate holonomy to the geodesics
by projectively pasting with the boundary components of the split surface. 
(See \cite{SulTh} and \cite{cdcr2} for details.) 
This construction is called {\em grafting}. 
Such an annulus is a union of elementary annuli meeting in geodesic boundary 
components. 

By Proposition 4.5 of \cite{cdcr2}, 
each boundary component of the inserted annulus is principal as we used simple closed geodesics in a properly convex closed $\rptwo$-surface $\Sigma'$.
Each inserted annulus has a domain $K$ in $\SI^2_{\infty}$ covering it since $\mathcal{D}'$ is a domain in $\SI^2_{\infty}$. 
The removed simple closed geodesics correspond to $\mathbf{l}_i$ and $\mathbf{l}_{i-}$ for $i \in {\mathcal{J}}$. 
$K$ must have boundary arcs ${\mathbf{l}}_i$ and ${\mathbf{l}}_{i -}$ for some $i \in {\mathcal{J}}$.
Since $K - \Ss_{+} - \Ss_{-}$ is a domain in $\SI^2_{\infty}$ and covers the compact annulus with boundary, 
$K - \Ss_{+} - \Ss_{-}$ is a union of components of 
\[\SI^{2}_{\infty} - (\Ss_{+} \cup \Ss_{-})  - (l_{1} \cup l_{2} \cup l_{3}) \] 
and components of \[(l_{1} \cup l_{2} \cup l_{3})  - \{a, a_{-}, r, r_{-}, s, s_{-}\}\]
using the notation of $\vth'$ of Section \ref{subsec:rpt} adopted to $\bg_{i}$. 
By Goldman \cite{Gthesis}, 
we can show that $K=\mathcal{R}_i$ or $K = \mathcal{R}_{i-}$ for some $i \in \mathcal{J}$.
(See also \cite{SulTh}.)
In order for these domains to be disjoint, they are all of from $\mathcal{R}_i$ or 
all of form $\mathcal{R}_{i-}$ for $i \in \mathcal{J}$. 
Now, it follows that $\mathcal{D}$ or $\mathcal{D}_{-}$ are only possible domains 
covering closed $\rptwo$-surfaces with the holonomy groups equal to $\L(\Gamma)$. 
(See also \cite{chgo}.)
\end{proof}


\section{$\Gamma$ acts properly on $\Lspace \cup \tSigma$.} \label{sec:propdisc}

First we briefly summarize the relevant results of Goldman, Labourie, and Margulis in \cite{GLM} and \cite{GL}. 
We will need some parts of this section for our main proof. 
\subsection{Diffused Margulis invariants}\label{sub:Marg}

We assume that 
$\Gamma$ acts freely and properly on $\Lspace$ and hence on $\Ss_+$ (or equivalent on $\Ss_-$).
$\Gamma$ acts on $\Ss_+$ as a geometrically finite Fuchsian group.


%
If $\Gamma$ acts freely and properly on $\Lspace$, then
Margulis invariants of nonidentity elements are all positive or all negative by  the opposite sign lemma. (See \cite{GLM}.)
By choosing the opposite orientation of $\Lspace$, 
we can change the all signs of the invariants.
We henceforth assume that the Margulis invariants of nonidentity elements are all positive.

The Margulis invariant $\mu(c)$ of a current $c$ supported on a 
simple closed geodesic $\hat c$ on $\Ss_{+}/\Gamma$ is given by $\mu(g)/l_{\Ss_{+}}(\hat c)$
for the corresponding element $g \in \Gamma$
and the hyperbolic length $l_{\Ss_{+}}(\hat c)$. 
In the second main theorem of \cite{GLM}, they 
define the diffused Margulis invariants  on $\mathcal{C}(\Ss_{+}/\Gamma)$ 
by continuously extending the Margulis invariants. 

These invariants of elements of ${\mathcal C}(\Ss_+/\Gamma)$ are positive by the second main 
theorem in \cite{GLM}.  The subspace of currents supported on
closed geodesics with total measure $1$ 
is precompact in ${\mathcal C}(\Ss_+/\Gamma)$.
Thus, there exists $C>0$ such that
\begin{equation} \label{eqn:lb} 
1/C \leq \frac{\mu(g)}{l_{\Ss_+}(g)} \leq C \hbox{ for } g \in \Gamma - \{\Idd \}
\end{equation}
for all $g\in\Gamma$, 
where $l_{\Ss_+}(g)$ denotes the minimum of the hyperbolic distances between $x$ and $g(x)$ for 
$x \in \Ss_+$.

%

\subsection{Neutralized sections}\label{sub:netralized}

We construct a flat affine bundle 
$\mathbf E$ over the unit tangent bundle $\Uu\Sigma_+$ of $\Sigma_+$ by
taking the quotient of $\Lspace \times \Uu\Ss_+$  by the diagonal action  
\[\gamma(v, x) = (\gamma(v), D\gamma(x)), v \in \Lspace, x \in \Uu\Ss_+, \gamma \in \Gamma\]
where $D\gamma$ is the differential map.  
%
The cover of $\mathbf E$ is denoted by $\hat{\mathbf E}$ and is identical with 
$\Lspace \times \Uu\Ss_+$. We denote by 
\[\pi_{\Lspace}: \hat{\mathbf E} = \Lspace \times \Uu\Ss_+ \ra \Lspace\]
the projection.

We define $\mathbf V$ as the quotient space 
of $\V \times \Uu\Ss_+$  under the diagonal action 
\[\gamma(v, x) = (\L(\gamma)(v), D\gamma(x)), v \in \V, x \in \Uu\Ss_+, \gamma \in \Gamma.\]

Given $\vu \in \Uu \Ss_+$, let $l(\vu)$ denote the geodesic tangent to $\vu$ in 
the hyperbolic plane $\Ss_+$ and $\partial_+ l(\vu) \in \partial \Ss_+$ the starting point 
and $\partial_- l(\vu) \in \partial \Ss_+$  the endpoint of $l(\vu)$. 
The maps $\Uu \Ss_+ \ra \partial \Ss_+$ given by $\vu \ra \partial_\pm l(\vu)$ are smooth maps. 
We define
\[\tilde \bnu: \vu \in \Uu\Ss_+ \mapsto \left(\vu, \frac{\rho(\vu) \times \alpha(\vu)}{|||\rho(\vu) \times \alpha(\vu)|||}\right) \] 
for the null vectors $\rho(\vu)$ and $\alpha(\vu)$ 
where  $||| \cdot |||$ is the absolute value of the Lorentzian norm and 
\[\llrrparen{ \rho(\vu) }= \partial_+ l(\vu) \in \partial \Ss_+ \hbox{ and } 
\llrrparen{ \alpha(\vu) }= \partial_- l(\vu) \in \partial\Ss_+\]
hold. 
Since  \[{\L(\gamma)}(\tilde \bnu(x)) = \tilde \bnu(D\gamma(x)), x\in \Uu\Ss_+, \gamma \in \Gamma,\]
$\tilde \bnu$ induces a {\em neutral section} 
\[ \bnu: \Uu\Sigma_{+}\ra \mathbf V.\]
(See \S 4.2 of \cite{GLM} for details.)

%
%
%

Let $\nabla$ be the flat connection on ${\mathbf E}$
as a bundle over $\Uu \Ss_+$ with fiber isomorphic to $\Lspace$
 induced from the product structure of $\hat{\mathbf E} = \Lspace \times \Uu\Ss_+$. 
(See \S 3.2 and 3.3 of \cite{GLM}.)
Lemma 8.4 of \cite{GLM} finds 
a {\em neutralized section} ${\mathcal N}: \Uu_{\rm rec} \Sigma_+ \ra {\mathbf E}$ 
satisfying 
\begin{equation}\label{eqn:fnu} 
\nabla_X {\mathcal N} = f \bnu
\end{equation}
where $X$ is the vector field of geodesic flow on 
$\Uu_{\rm rec} \Sigma_+$ and $f$ is a positive valued function defined on $\Uu_{\rm rec} \Sigma_+$.
%
%
\subsection{Lifting the neutralized section to the coverings} \label{sub:liftcov}
Let $\Uu_{\rm rec} \Ss_+$ denote the inverse image of $\Uu_{\rm rec} \Sigma_+$ in 
$ \Uu\Ss_+$. 
Thus, we find the section $\tilde{\mathcal{N}}: \Uu_{\rm rec} \Ss_+ \ra \hat{\mathbf{E}}$
lifting $\mathcal{N}$ satisfying 
\begin{equation} \label{eqn:tildeN} 
\tilde{\mathcal{N}} \circ \gamma = \gamma \circ \tilde{\mathcal{N}}, \gamma \in \Gamma.
\end{equation}

%
Equation \eqref{eqn:fnu} also lifts to 
\begin{equation} \label{eqn:tfnu}
\nabla_X \tilde{\mathcal{N}} = \tilde f \tilde \bnu
\end{equation}
where $X$ is the unit-vector field of geodesic flow on $\Uu \Ss_+$,
$\nabla$ is the flat connection on $\Lspace \times \Uu \Ss_+$, 
and $\tilde f$ is a positive-valued function. 

\subsection{Mapping the convex core\/}
%


Let $\Uu\Lspace$ denote the space of unit spacelike vectors at each point of $\Lspace$
and $\Uu \Lspace/\Gamma$ the space of unit spacelike vectors at each point of $\Lspace/\Gamma$. 
Denote by $\Uu_{\rm rec} \Lspace$ in $\Uu \Lspace$ the 
inverse image of the subset of $\Uu_{\rm rec} \Lspace/\Gamma$ composed of 
unit tangent spacelike vectors tangent to nonwandering  spacelike geodesics.


We give the topology of the space $\mathcal{G}_{\rm rec}\Ss_+$ of oriented geodesics in $\Ss_+$ 
mapping to nonwandering  oriented geodesics in $\Sigma_+$ 
by the quotient topology from $\Uu_{\rm rec} \Ss_+$. 
Similarly, we give the topology of the space $\mathcal{G}_{\rm rec}\Lspace$ 
of spacelike oriented geodesics in $\Lspace$ mapping to nonwandering  geodesics in $\Lspace/\Gamma$ 
by the quotient topology from $\Uu_{\rm rec} \Lspace$. 

A {\em bounded subset} of $\mathcal{G}_{\rm rec} \Ss_+$ is 
a set of geodesics passing a bounded set in $\Uu \Ss_+$, 
and a {\em bounded subset} of $\mathcal{G}_{\rm rec} \Lspace$ is 
a set of spacelike geodesics passing a bounded set in $\Uu \Lspace$
under the Euclidean metric $d_E$.

%
%
\begin{prop}[Theorem 1 of \cite{GL}] \label{prop:corr} 
The map $\tilde {\mathcal N}$ induces 
a  continuous function 
\[\mathcal{G}_{\rm rec}\Ss_+ \xrightarrow{\mathscr{N}}
\mathcal{G}_{\rm rec}\Lspace\]
where
a bounded set of elements of $\mathcal{G}_{rec}\Ss_+$ maps 
to a bounded set in $\mathcal{G}_{rec}\Lspace$. 
\end{prop}
\begin{proof} 
$\pi_{\Lspace} \circ \tilde {\mathcal N}$ maps 
a flow segment in $\Uu_{\rm rec} \Ss_+$ to a spacelike line segment in 
$\Lspace$. 
For a flow segment in $\Uu_{\rm rec} \Ss_+$, 
$\pi_{\Lspace} \circ \tilde {\mathcal N}$ sends it to a spacelike line segment in $\Lspace$. 
Since the geodesic flow transfers to the geodesic flow on $\Lspace$, 
this induces $\mathscr{N}$. 

The next statement follows by the fact that $\mathscr{N}$ is induced by the continuous map
$\pi_{\Lspace} \circ \tilde {\mathcal N}$ which sends a compact set to a compact set. 



\end{proof}

%
%


\subsection{The proof of proper discontinuity}
Let $\Gamma$ be as above. 
We assume initially that 
$\L(\Gamma) \subset \SOto^o$ and $\Ss_+/\Gamma$ is 
a complete genus $\tilde{\bg}$ hyperbolic surface with $\mathsf b$ ideal boundary components.
As above, we assume that the Margulis invariants of nonidentity elements are all positive.

By Theorem~\ref{thm:domainq},
\begin{equation} \label{eqn:tsigma2}
{\tSigma}\ :=\ \Ss_+ \cup\;  
 \coprod_{i\in {\mathcal{J}}} {\mathcal{R}}_i\;  \cup  \Ss_- \; \subset\; \SI^2_{\infty} 
\end{equation}
is a $\Gamma$-invariant open domain. 
The quotient map $\tSigma\longrightarrow\Sigma$ is a covering space
onto a closed $\rptwo$-surface $\Sigma$. 
This surface decomposes as 
a union of $\Sigma'_+$ and $\Sigma'_-$ and annuli 
$A_i$ for $i=1, \ldots, {\bbb}$. 






\begin{prop}\label{prop:propdisc} 
Let $\Gamma$ be 
as above.
Assume ${\L}(\Gamma) \subset \SO(2, 1)^o$. 
Then $\Gamma$ acts freely and properly on 
$\Lspace \cup \tSigma$. 
\end{prop} 
\begin{proof} 

Since $\Gamma$ acts freely 
on both $\R^3 = \Lspace$ and on $\tSigma$, it 
acts freely on the union. 
Since $\Gamma$ is discrete, it suffices to show that the action of 
$\Gamma$ on $\Lspace\cup \tSigma$ is proper.



Suppose that $\Gamma$ contains a sequence 
$g_n$ 
and that 
$K \subset \Lspace \cup \tSigma$ is compact
so that 
\begin{equation} \label{eqn:univ}
g_n(K) \cap K \ne \emp
\end{equation}
for all $n$.
We  show that the sequence $g_n$ 
is finite. 

We prove by contradiction: 
Suppose that $\{g_n\}$ is an infinite sequence of mutually distinct isometries. 


Recall that $\Gamma$ acts on 
$\partial\Ss_+$ 
as a {\em convergence group.} 
Let $a_n$ and $r_n$ denote the attracting fixed point and the repelling fixed point of $g_n$ respectively. 
We can choose a subsequence $g_{n_k}$ satisfying the convergence group property (see \S \ref{subsec:dyn}). 
For the attracting fixed point $a_{n_k}$ and the repelling fixed point $r_{n_k}$ in $\partial \Ss_+$ of $g_{n_k}$,
Lemma \ref{lem:fix} 
implies:
\begin{align*}
\{a_{n_k}\} & \ra a \\ 
\{r_{n_k}\}  & \ra b 
\end{align*}
for $a, b \in \partial \Ss_+$ provided $a \ne b$. 


{\em First, we will consider the case when $a \ne b$}: 

\subsection{Convergences of the axes} 
Now letting $\{g_n \}$ denote the subsequence, we note that each
$g_n$ has an attracting fixed point $a_n$ and 
a repelling fixed point $r_n$ in $\partial\Ss_+$.
By our conditions,  $\{a_n\} \ra a$ and $\{r_n\} \ra b$ hold by Lemma \ref{lem:fix}. 

We define 
\[ \nu := \frac{\beta \times \alpha}{|||\beta \times \alpha |||}\]
of nonzero vectors $\alpha$ and $\beta$ corresponding to $a$ and $b$ respectively. 
A Lorentzian isometry element $g_n$ acts as a translation on a unique spacelike line $\Axis({g_n})$ in the direction $\nu_n$ of
eigenvalue $1$. 
Let $\alpha_n$ and $\rho_n$ denote the null vectors in the directions of $a_n$ and $r_n$ respectively
so that $\alpha_n \ra \alpha$ and $\rho_n \ra \beta$ hold.
(We can assume without loss of generality that 
\[|||\rho_n \times \alpha_n ||| = 1, |||\beta \times \alpha ||| = 1.)\]
We let $\nu_n $ 
be the cross-product of $\rho_n$ and $\alpha_n$; that is, $\nu_n = \vv_0(g_n).$
Thus  $\llrrparen{ \nu_n} \ra \llrrparen{ \nu} \in \SI^2_{\infty}$.
Since $\{a_n\} \ra a$, 
the sequence
\[ \ovl{a_n \llrrparen{ \nu_n }  a_{n, -}}= \clo(\varepsilon(a_n)) \]
converges to a segment 
\[ \ovl{a \llrrparen{ \nu } a_-}= \clo(\varepsilon(a)). \]

Since the geodesics with endpoints  
$a_n, r_n$ pass the bounded part $K''$ of the unit tangent bundle of $\Ss_+$, 
\[\clo(\Axis({g_n})) = \clo({\mathscr{N}}(\overrightarrow{r_na_n}^o)) 
\; \longrightarrow \;  %
\clo({\mathscr{N}}(\overrightarrow{ba}^o))\] 
by the continuity of $\mathscr{N}$ of Proposition \ref{prop:corr}. 
Hence, each $\Axis({g_n})$ passes a point $p_n$, and $\{p_n\}$ forms a convergent sequence in $\Lspace$.
By choosing a subsequence, we assume without loss of generality that $p_n \ra p_\infty$
for $p_\infty \in \Lspace$.

Recall that the span of $\Axis(g_n)$ and $r_n$ is  
\[S_n :=  \W^{ws}(g_n). \] 
{\em To conclude}\/ (I)(i), {\em we obtained a sequence $\{g_n\}$ satisfying the properties}\/:  
\begin{align} \label{eqn:giprop}
& \alpha_n \ra \alpha,  \rho_n \ra \beta, \nu_n 
\ra \nu, \nonumber \\ 
& a_n \ra a, r_n \ra b, \, \llrrparen{ \nu_n  }\ra  \llrrparen{ \nu },  p_n \ra p_\infty, \nonumber \\ 
& \clo(\Axis(g_n)) \ra \, \ovl{\llrrparen{ \nu } p_\infty \, \llrrparen{ -\nu } },  \nonumber \\
& S_n 
\ra S_\infty.
\end{align} 

%

The last property follows easily.

%



\subsection{Conjugating to a standard boost form}\label{sec:ToStandardBoost}

Next we find a uniformly bounded sequence $h_n$ of coordinate changes
so that $h_ng_nh_n^{-1}$ is in standard form \eqref{eq:StandardProjectiveBoostForm}.
(See \S\ref{sec:StandardBoostForm}.) 
The uniform bounds on $h_n$ follow from dynamical properties established 
in \cite{GL} as stated by equation \eqref{eqn:giprop}. 


We now introduce $h_n \in \Aut(\SI^3)$ used for coordinatizing $\SI^3$ for each $n$.  
We choose $h_n$ so that 
\begin{equation} \label{eqn:map}
h_n(a_n) =\be_{1},  h_n(\llrrparen{ \nu_n})  = \be_{2},
h_n(b_n)   = \be_{3},  
h_n(p_n) = \be_{4},
\end{equation}
and $\Axis({g_n})$ is sent to 
\[
\overline{~\be_{2}\be_{4}\be_{{2-}}}^o. 
\]
By \eqref{eqn:giprop} and Lemma \ref{lem:matrixcor}, $\{h_n\}$ can be chosen so that $\{h_n\}$ converges to
a bi-Lipschitz map $h \in \Aut(\SI^3)$, uniformly in the $C^s$-sense for any integer $s \geq 0$. 

We can consider $\vv_0(g_n) = \nu_n$ as a tangent vector at the origin $O$ in $\Lspace$. 
Since $h_n$ acts on $\Lspace$, $h_n$ is an affine transformation and hence 
the linearization $\L(h_n)$ equals the differential $D h_{n, x}$ at any $x \in \Lspace$. 
By \eqref{eqn:map}, $\L(h_n)$ sends the unit Lorentzian norm vector $\vv_0(g_n)$ to a vector parallel to $(0, 1, 0) \in \V$.
By post-composing with a projective map $\Phi_n$ fixing $\be_n$, $i=1, \dots, 4$, 
we further modify $h_n$ so that  
\begin{equation} \label{eqn:Lhi} 
D h_{n, p_n}(\nu_n) = (0, 1, 0)_{\be_4} \hbox{ and hence } 
\L(h_n)(\vv_0(g_n)) = (0, 1, 0). 
\end{equation} 


Since $\{\vv_0(g_n)\}$ is a convergent sequence of vectors, we proved  
\begin{lem} \label{lem:hn} 
We can choose the sequences $h_n$ satisfying equation \eqref{eqn:map}
 convergent uniformly in $C^{s}$ to $h \in \Aut(\SI^3)$ for any integer $s \geq 0$
 so that 
\begin{alignat}{2} \label{equ:hiC}
&C^{-1} \bdd(h_n(x), h_n(y))   & \;\leq\;  \bdd(x, y) & \;\leq\;  C \bdd(h_n(x), h_n(y)) \nonumber \\ 
&C^{-1} \bdd(h_n^{-1}(x), h_n^{-1}(y))\  & \;\leq\;  \bdd(x, y) & \;\leq\;  C \bdd(h_n^{-1}(x), h_n^{-1}(y)) 
\end{alignat} 
for all $x, y \in \SI^3, n = 1, 2, \dots$ and a fixed positive constant $C$.
Furthermore, the sequence $h_n$ converges to $h \in \Aut(\SI^3)$ where 
\begin{align*}
h(a) & = \be_{1},  \\
 h(b)  &= \be_{3}, \\
 h(\llrrparen{ \nu }) & = \be_{2}, \\ 
 h(p) &= \be_{4}, \\
 \L(h)(\nu) &= (0, 1, 0).
 \end{align*} 
\end{lem} 

\subsection{Normalization}

We can conjugate by a sequence of $h_n$ of orientation-preserving bi-Lipschitz maps
considered as an element of $\Aut(\SI^3)$ so that 
$h_n g_n h_n^{-1}$ is of form 
\begin{equation}\label{eqn:giconj}
\hat g_n :=  
\bmatrix 
 \lambda(g_n) & 0 & 0  & 0\\ 
0                  & 1 & 0     &\mu(g_n) \\ 
0                  & 0 & \lambda(g_n)^{-1} & 0 \\     
0                  & 0 & 0    & 1 
\endbmatrix
\end{equation}
where $\lambda(g_n) > 1$ and the orientation preserving
$h_n$ is given by sending $\be_{1}$ to $a_n$,
$\be_{2}$ to $\llrrparen{ \nu_n } $, $\be_{3}$ to $r_n$, and $\be_{4}$ to $p_n$. 
Also, 
\begin{equation} \label{eqn:hatgi}
g_n = h_n^{-1} \circ \hat g_n \circ h_n;
 \end{equation} 
hence, $g_n$ acts as the standard form matrix in \eqref{eqn:giconj} does
up to coordinatization by $h_n$.

The closed geodesic $c_n \subset \Sigma_+$ corresponding to $g_n$
has length $l_{\Ss_{+}}(g_{n}) = 2\log(\lambda(g_{n}))$ for the largest eigenvalue $\lambda(g_{n})$ of $g_{n}$.
Recall that $\{g_n\}$ has a property \eqref{eqn:giprop}:  
\begin{lem} \label{lem:kiinf} 
%
$\lambda(g_n) \longrightarrow +\infty, \quad \mu(g_n) \longrightarrow +\infty$, 
and  $\frac{\mu(g_n)}{\lambda(g_n)} \longrightarrow 0$
as $i\longrightarrow \infty$.
\end{lem}
\begin{proof} 
Since $\{g_n\}$ restricts to a sequence of mutually distinct hyperbolic isometries in $\Ss_+$
with sequences of fixed points $\{a_n\} \ra a$ and $\{r_n\} \ra b$ and $a \ne b$, it follows  that $\{\lambda(g_n)\} \ra \infty$
by the discreteness of $\L(\Gamma)$.   
By \eqref{eqn:lb}, 
$\mu(g_n) \ra \infty$.  
By the main corollary of \cite{GLM}, $\mu(g_n) \leq C' l_{\Ss_{+}}(g_n)$
and hence $\mu(g_n) \leq 2C' |\log \lambda(g_n)|$ hold for $C' >0$; 
we obtain the final limit.  
\end{proof}

The set
\[ \hat \Lambda := \He - (\Lspace \cup \tSigma)\]
is compact. 
Moreover, 
\[ \hat \Lambda = \bigcup_{z \in \Lambda} \clo(\varepsilon(z)) \] 
for the limit set $\Lambda$ of $\Gamma$. 
 We recall from Lemma \ref{lem:Pi0}:
 \begin{itemize}
\item the stable sphere $\SI^2_0$ given by $x = 0$,
\item  the segment $\eta_+\  :=\  \ovl{~\be_{1}\be_{2}\be_{1-}~}$,  
\item the segment $\eta_-:= \ovl{~\be_{3} \be_{{2-}}\be_{3-}~}$.
\end{itemize}
Then:
\begin{itemize}
\item $h_n^{-1}(\eta_-) = \clo(\varepsilon(r_n))$ equals 
$\ovl{r_n \llrrparen{ -\nu_n}  r_{n-}}  \subset\  \hat \Lambda$ and
\item $h_n^{-1}(\eta_+) = \clo(\varepsilon(a_n))$ equals 
$\ovl{a_n \llrrparen{ \nu_n }  a_{n-}}  \subset  \hat \Lambda$.
\end{itemize}

\subsection{The convergence of convex balls}
Cover $K$ by compact convex sets 
and consider how their $g_n$-images transform.
By \S\ref{sec:ToStandardBoost}, 
we may assume by choosing a subsequence of $g_n$ satisfying 
\begin{itemize}
\item $S_n  \ra S_\infty$ a great $2$-dimensional sphere, 
\item $h_n^{-1}(\eta_-) \ra \eta^\infty_- = \clo(\varepsilon(b)) \subset \hat \Lambda$, and 
\item $h_n^{-1}(\eta_+)  \ra \eta^\infty_+ = \clo(\varepsilon(a)) \subset \hat \Lambda$. 
\end{itemize}
Here 
\begin{itemize}
\item $\eta^\infty_+$ is a segment $\ovl{a\llrrparen{ \nu} a_-} \subset \hat \Lambda$ and 
\item $\eta^\infty_-$ is a segment $\ovl{b \llrrparen{ -\nu} b_-} \subset \hat \Lambda$.
\end{itemize} 
They are all tangent to 
$\partial \Ss_+$ and
%
\[ K \; \subset\; \Lspace \cup \tSigma \;=\;  \He - \hat \Lambda
\quad\subset\quad \He - \bigcup_{n=1}^\infty 
h_n^{-1}(\eta_-). \]


\noindent
Let $\He_a$ and $\He_{a_-}$ denote the respective components of 
$\He - S_\infty$ containing $a$ or $a_-$.
We %
cover $K$ by balls $C_j$ ($j\in\mathfrak{J}$) bounded away from $S_\infty$ and  
by balls $B_l$ ($l\in\mathfrak{L}$) meeting $S_\infty$ precisely at the centers. 
\bigskip

\noindent
First, we  cover $K$ as follows:

\begin{itemize} 
\item 
$K_\infty := K \cap S_\infty$. 
is a compact subset of an open $3$-dimensional hemisphere in 
$S_\infty - \eta^\infty_-$. 
Cover $K_\infty$ by the interiors $\widetilde{B_l^o}$ of 
finitely many balls 
\[ \widetilde{B_l} \subset \Lspace \cup \tSigma\] 
with centers in $S_\infty$ for $l \in \mathfrak{L}$ of small $\bdd$-radius $\eps_0 > 0$ 
in $\He - \hat \Lambda$. 
(We assume that balls $\hat B_l$ of $\bdd$-radius $2\eps_0$ with the same centers remain in $\He - \hat \Lambda$.)
\item Choose a foliation in $\He$ by perpendicular segments to $S_\infty$. 
We divide the ball $\widetilde{B_{l}}$ into the components $B^+_l, B^-_l$ of $\widetilde{B_l} - S_\infty$ 
for each $l \in \mathfrak{L}$. 
Using isotopies $i_{+, l}$ and $i_{-, l}$ preserving the foliation so that 
\begin{itemize} 
\item the images    
\begin{alignat*}{3}
C^+_l && \subset \He_a && \hbox{ for } C^{+}_{l}:= i_{+, l}\left(\clo(B^{+}_{l})\right),  \\
 C^-_l && \subset \He_{a_-} && \hbox{ for } C^{-}_{l}:= i_{-, l}\left(\clo(B^{-}_{l})\right), 
\end{alignat*}
are compact and convex, and
\item 
the {\em dumbbell\/} $\Join(C^+_l, C^-_l)$ 
remains in $\hat B_l \subset \Lspace \cup \tSigma$
by displacing exclusively in  the $2\eps_0$-balls. 
\end{itemize} 
\item Furthermore
\begin{alignat}{3}
K_\infty  & \subset  \bigcup_{l \in \mathfrak{L}} \Join(C^+_l, C^-_l)^o \\  
& \subset \Lspace \cup \tSigma, 
\end{alignat}
\noindent
and for all $l \in \mathfrak{L}$,   
\begin{align}
\bdd\Big(C^+_l, S_\infty\Big) & > \delta_0, \notag\\
\bdd\Big(C^-_l, S_\infty\Big) & > \delta_0  
\label{eqn:Ci}  
\end{align}  
for some $\delta_0> 0$. 
\item We obtain a compact subset
\begin{align*}
 K' &:= K - \bigcup_{l \in\mathfrak{L}} \Join(C^+_l, C^-_l)^o  \\ 
 &\quad \subset \Big(\Lspace \cup \tSigma\Big) - S_\infty \\ 
 & \qquad = \He -  \Big( \hat \Lambda \cup S_\infty\Big).
 \end{align*}
Let $\eps_{K'} := \bdd(S_\infty, K')$. 
Find a finite set of $\bdd$-balls $B_j$, where $j \in \mathfrak{J}$, such that:
\begin{itemize}
\item $K' \subset \bigcup_{j \in \mathfrak{J}} B_j$;
\item $B_j$ has radius  $\frac{\eps_{K'}}{2}$; 
\item \begin{equation}\label{eqn:Bi} 
\bdd(B_j, S_\infty) \geq \frac{\eps_{K'}}{2}.
\end{equation}
\end{itemize}\end{itemize}
\end{proof}

\noindent Divide
$\mathfrak{J}$ into $\mathfrak{J}^+$ and $\mathfrak{J}^-$ so that 
$j \in \mathfrak{J}^+$ if and only if $B_j \subset \He_a$, and
$j \in \mathfrak{J}^-$ if and only if $B_j \subset \He_{a_-}$.
\subsection{The balls which are bounded away from $S_\infty$}



\begin{lem} \label{lem:BddAwayFromPlane}
Let $g_n$ be a sequence as above and let
\[ \mathscr{B} \ := 
\ \bigcup_{j\in \mathfrak{J}^+} B_j \;\cup\; \bigcup_{j \in \mathfrak{J}^-} B_j
\;\cup\; \bigcup_{l \in \mathfrak{L}} \big( C^+_l \cup C^-_l \big). \]
Then as $n\longrightarrow\infty$,  the sequence of sets
$\{g_n(\mathscr{B})\}$ converges to $\{a, -a\}$,
or to one of its subsets $\{a\}$ or $\{-a\}$. \end{lem}
\begin{proof}
For sufficiently small $\delta >0$, there exists $N_0$ such that for 
$n> N_0$
\[ \bdd\left(S_n, \mathscr{B} \right) > \delta \]
by \eqref{eqn:Ci} and \eqref{eqn:Bi} since $S_n \ra S_\infty$.

Since  $h_n(S_n) = \SI^2_0$,  
\eqref{equ:hiC} implies 
\[ 
\bdd\left(\SI^2_0, h_n(\mathscr{B})\right)  > C^{-1} \delta 
\]
 for $n > N_0$. 
 We decompose $\mathscr{B} = \mathscr{B}^{+}\cup \mathscr{B}^{-}$ 
 where 
 \[\mathscr{B}^{+} := \bigcup_{j\in \mathfrak{J}^+} B_j \cup \bigcup_{l \in \mathfrak{L}}  C^+_l
\hbox{ and } \mathscr{B}^{-} := \bigcup_{j\in \mathfrak{J}^-} B_j \cup \bigcup_{l \in \mathfrak{L}}  C^-_l.\] 
Now act by $\hat g_n$ and apply  
Lemma \ref{lem:Pi0} (a):
For any $\eps''> 0$, there exists $N_1$ so that for $n > N_1$, 
\[ \hat g_n  h_n(\mathscr{B}^{+}) \subset N_{\eps''}(\be_{1}), \,
\hat g_n  h_n(\mathscr{B}^{-}) \subset N_{\eps''}(\be_{1-})\] hold.  
\item Hence, for every $\eps'' > 0$, 
there exists $N_1 > N_0$ so that for $n > N_1$, 
the sequence of the images of these sets under $g_n$ go into the $\eps''$-$\bdd$-neighborhoods of $a_n$ or $a_{n -}$ by  \eqref{equ:hiC}. 
Since $a_n \ra a$ and $a_{n-} \ra a_-$, 
\begin{equation*}
\bdd^H\left( \{ a, a_-\},  g_n(\mathscr{B})
\right) \ra 0
\end{equation*} 
when $\mathscr{B}^{+}$ and $\mathscr{B}^{-}$ are both nonempty. 
The other cases are similar.   
\end{proof}

\noindent
If $K_\infty = \emp$, 
then the reader can proceed to \S \ref{sub:concl}. 
Thus assume $K_\infty \ne \emp$ from now on. 



\subsection{The balls meeting $S_\infty$.} 
We show that the dumbbells all converge to the semicircle outside of $\Lspace \cup \tilde \Sigma$.
\begin{lem}
For each $l\in\mathcal{L}$, 
the images $g_n\big(\Join(C^+_l, C^-_l)\big)$ converge to
the great semicircle $\eta^\infty_+ = \ovl{a \llrrparen{ \nu } a_-}.$
\end{lem}
\begin{proof}

Since $\bigcup_{l \in \mathcal{I}} \hat B_l$ and $\hat \Lambda$ 
are disjoint compact subsets of $\He$,
\[ 
\bdd\bigg( \bigcup_{l \in \mathcal{L}} \Join(C^+_l, C^-_l),\  ~\hat \Lambda\bigg) > \delta'
\]
for some $\delta'>0$.
We define \[C_{l, n} := S_n \cap \Join(C^+_l, C^-_l)\] for $l \in \mathcal{L}$. 
Since $\clo(\varepsilon(r_n))$ lies in the compact set $\hat\Lambda$, 
\[
\bdd\Big(C_{l, n} ,  \clo\big(\varepsilon(r_n)\big)\Big) > \delta' \]
for a fixed $\delta' > 0$. 
Thus by \eqref{equ:hiC}, 
\[ \bdd\big(h_n(C_{l, n}), \eta_{-}\big) > C^{-1}\delta' \]
for all $l \in\mathcal{L}$. 
By Lemma \ref{lem:Pi0},  
for every $\eps > 0$, there exists  $N_2$ (independent of $l$) 
such that 
\[ \hat g_n \circ h_n(C_{l, n}) \subset N_{\eps}(\be_{2})\]  
for  $n > N_2$. 
By \eqref{equ:hiC},  
\[ g_n (C_{l, n}) \subset N_{C\eps}(\llrrparen{ \nu_n}) \]
for $n > N_2$. 
Therefore \eqref{eqn:hatgi} implies 
\[
\lim_{n\ra\infty} g_n(C_{l, n}) = \llrrparen{\nu}. \] 

\noindent 
Lemmas~\ref{lem:geoconv}  and \ref{lem:BddAwayFromPlane} imply 
\[ \lim_{n\to\infty} g_n(\Join(C_{l}^{+}, C_{l, n}))
= \ovl{a\llrrparen{ \nu }} \]   
since
\[ g_n\big(\Join(C^+_l, C_{l, n})\big) = \Join\big(g_n(C^+_l), g_n(C_{l, n})\big). \] 
Similarly, 
\[ 
\lim_{n\to\infty} g_n\big(\Join(C^-_l, C_{l, n})\big)\ = \  \ovl{\llrrparen{ \nu} a_-}.\]
Since
\[
\Join(C^+_l, C_{l, n})  \cup \ \Join(C^-_l, C_{l, n})  = 
\Join(C^+_l, C^-_l), \]
and \[\ovl{a\,\llrrparen{ \nu }} \cup \, \ovl{\llrrparen{ \nu}  a_-}  =\eta^\infty_+,\]
facts in \S \ref{subsec:space} imply
\begin{equation}\label{eqn:JC}
\lim_{n\to\infty}
g_n\big(\Join(C^+_l, C^-_l)\big) =  \eta^\infty_+ 
\end{equation} 
for each $l \in\mathcal{L}$. 
\end{proof}


\subsection{Conclusion of the proof when attractors and repellers differ} \label{sub:concl} 
By our assumption of equation \eqref{eqn:univ}, $g_n(K) \cap K \ne \emp$ for all $n$.
There has to be some 
fixed pair of balls $B$ and $B'$ of one of the types:
\begin{itemize}
\item $B_j$ where $j \in\mathfrak{J}$;
\item $\Join(C^+_l, C^-_l)$ for $l \in\mathcal{L}$ 
\end{itemize}
so that $g_n(B) \cap B' \ne \emp$  for infinitely many $n$. 
However, Lemma~\ref{lem:BddAwayFromPlane} and \eqref{eqn:JC} imply
that $g_n(B)$ converges to one of:
$\eta^\infty_+ , \{a\},   \{a_-\} \subset \SI^2_{\infty}$.
Thus, for each $\eps >0$ and every ball $B$ in the collection indexed by 
$\mathfrak{L}\cup\mathfrak{J}$, for sufficiently large $n$,
\[N_\eps(\eta^\infty_+) \supset g_n(B)\]
by definition of the geometric convergence. 
However, since the compact sets $\eta^\infty$ and $B'$ are disjoint,
$N_\eps(\eta^\infty_+) \cap B' = \emp$ 
for any ball $B' $ indexed by $ \mathfrak{L} \cup \mathfrak{J}$ for sufficiently small $\eps > 0$, 
obtaining a contradiction.

\subsection{Conclusion of the proof when attractors and repellers coincide}
Now we consider the $a=r$ case. 
Since $\Gamma$ is a nonelementary Fuchsian group, 
choose an element $\gamma_0 \in \Gamma$ so that 
\[ \gamma_0 (a) \ne a. \] 


Consider the sequence $\{\gamma_0 g_i\} $. 
For each $\eps > 0$  and
a precompact neighborhood $U$ in $\partial\Ss_+ - \{a\},$ 
there exists an integer $I_0$ such that 
\[ \gamma_0 g_n(U) \subset N_\eps(\gamma_0(a)) \hbox{ for } n > I_0.\]
Now consider the sequence $g_n^{-1}\gamma_0^{-1}$.
Let $V$ be a precompact neighborhood in $\partial\Ss_+  - \{\gamma_0(a) \}$. 
Since $\gamma_0^{-1}(V)$ is 
a precompact neighborhood in $\partial\Ss_+ - \{a\}$ and for arbitrary $\eps > 0$, there exists $I_0$ such that 
\[ g_n^{-1}(\gamma_0^{-1}(V)) \subset N_\eps(a) \hbox{ for } n > I_0. \]
Thus, $\gamma_0(a)$ is the attractor point  
and $a$ is the repeller point of the sequence $\{\gamma_0 g_n\}$. (See \S \ref{subsec:dyn}.) 

Since 
\[ \gamma_0 g_n (K \cup \gamma_0(K)) \cap (K \cup \gamma_0(K)) \ne \emp
\]  
for infinitely many $n$, 
we are reduced to the case where $a \ne r$ 
by replacing $g_n$ with $\gamma_0 g_n$ 
and $K$ with another compact set 
\[ K \cup \gamma_0(K) \subset \Lspace \cup \tSigma. \]
 Again, we obtain contradiction.

\section{Tameness. }\label{sec:tameness}

\subsection{The compactification} 
Using the above notations, 
we note that $\tSigma/\Gamma$ is a closed surface of genus $\mathsf g$ and forms the boundary 
of the $3$-manifold $M := ( \Lspace \cup \tSigma)/\Gamma$ by Proposition \ref{prop:propdisc} 
provided ${\L}(\Gamma) \subset \SOto^o$. 
Therefore, $M$ is a $3$-manifold in general since $\Gamma$ has an index $\leq 2$
subgroup $\Gamma'$ with the property ${\L}(\Gamma') \subset \SOto^o$ and the fact that $\Gamma$ acts freely
on $ \Lspace \cup \tSigma$. 

We now show that $M$ is compact. The key idea is to isotopy 
some spheres into $\Lspace$ bounding 
compact $3$-balls. 
We first assume that $\L(\Gamma) \subset \SOto^o$
so that $\Gamma$ acts on $\Ss_+$ honestly. 

\begin{figure}[h]

\centerline{\includegraphics[height=6.5cm]{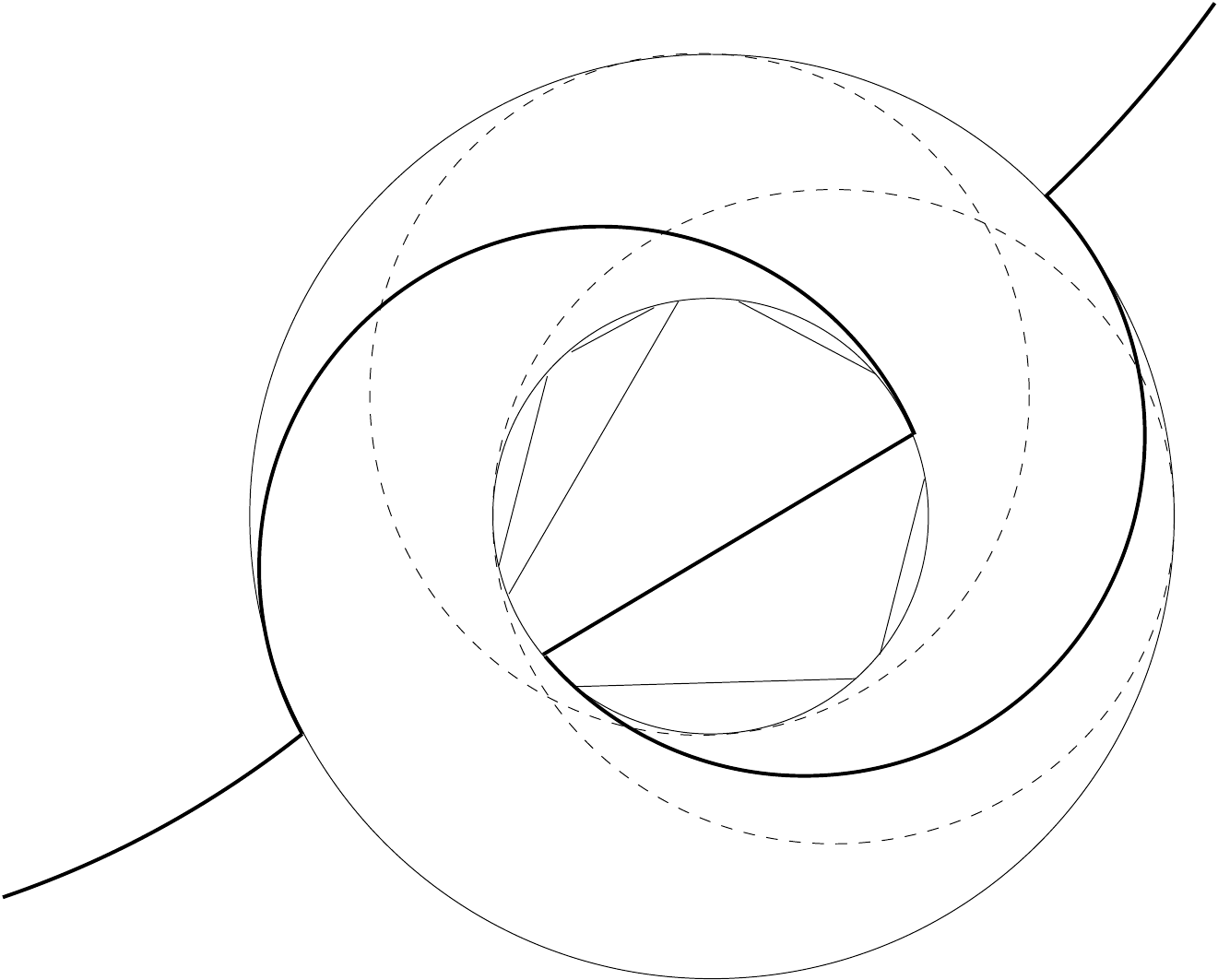}}
\caption{ The arcs in $\Ss_+$ and an example of $\hat \gamma_i$ in the bold arcs.\label{fig:decom}}

\end{figure}

\begin{prop} \label{prop:Dehn}
Each simple closed curve $\gamma$ in $\tSigma$ bounds a simple disk
in $\Lspace \cup \tSigma$. Let $c$ be a simple closed curve in $\Sigma$ 
that is homotopically trivial in $M$. Then $c$ bounds an embedded disk in $M= (\Lspace\cup \tSigma)/\Gamma$. 
\end{prop} 
\begin{proof} 
This is just Dehn's lemma. 
\end{proof}
We can find a collection of disjoint simple curves $\gamma_i$, $i \in {\mathcal{J}'}$, on $\tSigma$ 
for an index set ${\mathcal{J}'}$ so that the following hold:
\begin{itemize}
\item $\bigcup_{i\in {\mathcal{J}'}} \gamma_i$ is invariant under $\Gamma$.
\item $\bigcup_{i\in {\mathcal{J}'}} \gamma_i$ cuts $\tSigma$ into a union of open 
pair-of-pants $P_k$, $k \in K$, for an index set $K$.
The closure of each $P_k$ is a closed pair-of-pants. 
\item $\{ P_k \}_{k \in K}$ is a $\Gamma$-invariant set. 
\item Under the covering map $\pi: \tSigma \ra \tSigma/\Gamma$, 
each $\gamma_i$ for $i \in {\mathcal{J}'}$ maps to a simple closed curve in a one-to-one manner
and each $P_k$ for $k \in K$ embeds onto an open pair-of-pants. 
\end{itemize}

This is achieved by finding 
arcs in $\Ss_+/\Gamma$  cutting it into disks. 
Recall that $\Ss_+/\Gamma$ is an open surface compactified 
to the compact surface $\Sigma'_+$ with boundary.  We obtain 
a system of geodesic segments $\hat \alpha_i$, $i \in I$ for a finite set $I$, 
in $\Sigma'_+$ cutting $\Sigma'_+$ into a union of disks. Each disk is bounded by 
six arcs, alternating triple of which are arcs coming from $\partial \Sigma'_+$ 
and the other alternating triples are geodesic 
segments ending $\partial \Sigma'_+$.
We can assume that each $\hat \alpha_i$ 
is embedded in $\tSigma_+/\Gamma$ cutting it into finitely many hexagonal disks
so that each ideal boundary component of $\tSigma_+/\Gamma$ meets 
at least two arcs of form $\hat \alpha_i$. 

This gives a system of geodesic arcs $\alpha_j$, $j \in {\mathcal{J}'}$, 
in $\Ss_+$ for some infinite index set ${\mathcal{J}'}$ which decomposes $\tSigma_+$ into 
hexagonal disks and forms a $\Gamma$-invariant set. 
We define $\alpha_{j-}:= {\mathscr A}(\alpha_j)$ in $\clo(\Ss_-)$ for $j \in {\mathcal{J}'}$.  
We connect each endpoint $\partial_1(\alpha_j)$ and $\partial_2(\alpha_j)$ of $\alpha_j$ with 
its antipodal endpoint of $\alpha_{j-}$ by arcs of form $\varepsilon(x)$. 
\[ \gamma_j := \alpha_j \cup \alpha_{j-} \cup \varepsilon(\partial_1(\alpha_j)) \cup 
\varepsilon(\partial_2(\alpha_j)), j \in {\mathcal{J}'}.\]
(These form {\em crooked circles}.) 
We do this for each arc and obtain the above system of simple closed curves.
By construction, $\{\gamma_i\}_{i \in {\mathcal{J}'}}$ maps to a system of disjointly embedded curves 
$\hat \gamma_1, \ldots, \hat \gamma_{3{\bg}-3}$ in $\Sigma$, and each of them has a simple closed lift in $\tSigma$ and 
hence has a trivial holonomy. (See Figure \ref{fig:decom}.)
Also, the collection $\hat \gamma_1,\ldots,\hat \gamma_{3{\bg}-3}$ decompose the closed surface 
$\tSigma/\Gamma$ of genus $\bg$ for some ${\bg} \geq 2$ into $2{\bg}-2$ 
pairs of pants $P'_1, \ldots, P'_{2{\bg}-2}$.
(Here $\bg = 2\tilde \bg + {\mathsf b} -1$ for the genus $\tilde \bg$ of $\Ss_+/\Gamma$ and 
the number $\mathsf b$ of ideal boundary components.)

By trivial holonomy and  Dehn's lemma, each $\hat \gamma_i$ bounds a disk $D_i$ in $M$. By the $3$-manifold topology
of disk exchanges, we can choose $D_1, \ldots, D_{3\bg-3}$ to be mutually disjoint. (See \cite{Hempel}.)

Each pair of pants $P'_j$ and union with three adjacent ones in the collection $D_1, \dots ,  D_{3\bg-3}$ is homeomorphic 
to a $2$-sphere in $M$. 
We can lift the sphere into $\Lspace \cup \tSigma$ and being a sphere which can be pushed inside the cell $\Lspace$, 
it bounds a compact $3$-ball $B$ in $\Lspace \cup \tSigma$. 
The image of $B$ is a compact $3$-ball in $M$ as well since $\pi|\partial B$ is an embedding to a $2$-sphere
and the image of $B^o$ is disjoint from all other disks and $\pi$ is a covering map. 
Therefore, the Cauchy completion of each component of $M - D_1 - \cdots - D_{3\bg-3}$ 
is homeomorphic to a compact $3$-ball in $M$.  

Since $M$ is a union of the closure of the components that are $3$-balls identified with one another in disjoint disks, 
$M$ is a compact $3$-manifold. 
This implies also that $M$ is homeomorphic to a solid handlebody of genus $\mathsf g$; that is, 
our manifold $(\Lspace \cup \tSigma)/\Gamma$ is a compact $3$-manifold and 
so its interior $\Lspace /\Gamma$ is tame. This completes the proof of Theorem \ref{thm:B}
provided $\L(\Gamma) \subset \SOto^o$. 
(Also, it follows that the rank of $\Gamma$ equals $\mathsf g$.)

In the general case when ${\L}(\Gamma)$ is not a subgroup of $\SOto^o$, 
the above work applies to $(\Lspace \cup \tSigma)/\Gamma'$ 
double-covering $(\Lspace \cup \tSigma)/\Gamma$.
By Theorem 5.2 of Chapter 5 of \cite{Hempel},  $M$ is homeomorphic to a handlebody 
since $\Gamma$ is a free group and $(\Lspace \cup \tSigma)/\Gamma$ is
an aspherical compact $3$-manifold. 

We conclude: 

\begin{thm}[Compactification] \label{thm:C}
Let $\Lspace/\Gamma$ be a Margulis spacetime and $\Gamma$ is a Lorentzian isometry group without parabolics.  
Suppose that the Margulis invariants of elements of $\Gamma -\{\Idd\}$ are all positive. 
Choose $\tilde \Sigma$ by \eqref{eqn:tsigma1}. 
Then $(\Lspace \cup \tilde \Sigma)/\Gamma$ 
is homeomorphic to a compact handlebody. 
\end{thm} 
If we change the orientation or make the Margulis invariants to be all negative, then 
the conclusion holds for $(\Lspace \cup \mathcal{A}(\tilde \Sigma))/\Gamma$.
If we do both, then the conclusion holds for $(\Lspace \cup \tilde \Sigma)/\Gamma$.


\subsection{The almost-crooked-plane decomposition} \label{sub:acrooked}

A {\em crooked plane} is a closed disk embedded in $\Lspace$ constructed as follows: 
Take two null vectors $\vv_1$ and $\vv_2$ and respective parallel lines $L_1$ and $L_2$ meeting at a point $x$. 
Then they are on a timelike plane $P$. Take all timelike lines on $P$ from $x$ and take their union.
It is a union of two cones $C_1$ and $C_2$ with vertex at $x$. 
We obtain two null half-planes $\W(\vv_1)$ 
and $\W(\vv_2)$ containing $x$. 
A crooked plane is the union $C_1 \cup C_2 \cup \W(\vv_1) \cup \W(\vv_2)$.
(See \cite{Drumm_thesis} for details.)

A disk in $\Lspace$ is an {\em almost crooked plane} if it agrees with a crooked plane 
in the complement of its compact subset. Also, its immersed or embedded image is said to be an
{\em almost crooked plane} as well. 

\begin{prop} \label{prop:genschottky} 
Let $\tSigma$ and $\Gamma$ be as above.  
The Margulis spacetime $\Lspace/\Gamma$ without cusp has 
a system of disks with boundary in $(\Lspace \cup \tSigma)/\Gamma$ 
so that the closures of components of the complement are compact $3$-balls
and the disks are almost crooked planes. 
\end{prop} 
\begin{proof} 
First, assume ${\L}(\Gamma) \subset \SOto^o$. 
Each of our disks has the boundary of a crooked plane. 
In Dehn's lemma, one can arbitrarily assign the tubular neighborhood of the boundary of each disk 
as long as it is transversal to $\tSigma/\Gamma$. 
Thus, following the above discussions, we obtain the proof in this case. 

Suppose that ${\L}(\Gamma)$ is not in $\SOto^o$. 
Then we take a quotient $\Sigma' := \Ss_+/{\L}'(\Gamma)$ given by the projective action.
The nonorientable $\Sigma'$ admits a decomposition into hexagonal disks as above. 
We obtain the induced 
decomposition on $\Ss_+ \cup \Ss_-$ and we extend the above construction to this situation
and obtain the crooked circles and crooked planes. 

As above $(\Lspace \cup \tSigma)/\Gamma'$ double-covers 
$(\Lspace \cup \tSigma)/\Gamma$ with the deck transformation group generated 
by a projective automorphism $\phi$ of order two.
We choose a collection for the quotient $\Ss_+/{\L'}(\Gamma)$ as in the beginning of the proof 
of Proposition \ref{prop:Dehn}. Then $\phi$ acts on the system of circles by construction.  
We can modify each almost crooked plane $D$ not changing a neighborhood of $\partial D$ in $D$ 
so that $D \cap \phi(D) = \emp$ using the proof of Theorem 3 of \cite{GorL}. 
Hence, we obtain a collection of $\phi$-equivariant almost crooked planes that cuts $M$ into $3$-balls. 
\end{proof}

\begin{cor} 
A Margulis spacetime without cusp has a finite sided fundamental polyhedron in $\Lspace$. 
\end{cor}
\begin{proof} 
We take a union of finitely many of the $3$-balls obtained by cutting along the disks in $\Lspace$.
\end{proof}

\bibliographystyle{plain}


\end{document}